\documentclass[12pt]{article}

\usepackage{amsmath,amsfonts,amssymb, mathtools,amsthm}
\usepackage{color}
\usepackage[hmargin=2cm,vmargin=2cm]{geometry}

\usepackage{hyperref}
\usepackage{nameref,zref-xr}                    
\usepackage[all]{xy}   
\usepackage{tikz}
\usepackage{tikz-cd}
\usepackage{authblk}
\usepackage{caption}
\usepackage{subcaption}
\usepackage{array}

\usepackage{enumerate}

\newtheorem{theorem}{Theorem}[section]     
\newtheorem{proposition}[theorem]{Proposition} 
\newtheorem{lemma}[theorem]{Lemma} 
\newtheorem{corollary}[theorem]{Corollary}

\theoremstyle{definition}
\newtheorem{definition}[theorem]{Definition}
\theoremstyle{remark}
\newtheorem{remark}[theorem]{Remark}
\newtheorem{example}[theorem]{Example}

\newcommand{\mbP}{\mathbb P}
\newcommand{\mbA}{\mathbb A}
\newcommand{\mbZ}{\mathbb Z}
\newcommand{\mbQ}{\mathbb Q}
\newcommand{\mbC}{\mathbb C}
\newcommand{\mbR}{\mathbb R}
\newcommand{\Bl}{\mathrm{Bl}}
\newcommand{\prim}{\mathrm{prim}}

\newcommand{\ev}{\mathrm{ev}}
\newcommand{\vir}{\mathrm{vir}}
\newcommand{\virdim}{\mathrm{virdim}}
\newcommand{\Coef}{\mathrm{Coef}}
\newcommand{\Hcal}{\mathcal{H}}
\newcommand{\Zcal}{\mathcal{Z}}

\newcommand{\im}{\mathrm{im}}
\newcommand{\Aut}{\mathrm{Aut}}
\newcommand{\Ext}{\mathrm{Ext}}
\newcommand{\Hom}{\mathrm{Hom}}

\newcommand{\racine}{\mathrm{root}}

\title{Towards genus 0 Gromov--Witten theory of complete intersection of two quadrics via degeneration}

\author{Danil Gubarevich\thanks{
  Email: \texttt{danilphys180916@mail.ru}.}}

\numberwithin{equation}{section}

\begin{document}

\maketitle
\begin{abstract}

Let $X\subset\mbP^{m+2}$ be a smooth complete intersection of two quadrics of even dimension $m\geq 4$. We construct an orthogonal basis $e_1,\dots, e_{m+3}$ in primitive cohomology of $X$ in which the genus 0 degree $\frac{m}{2}$ Gromov--Witten invariant 
    \begin{align*}
    \left\langle  e_1, \dots, e_{m+3}\right\rangle_{0,m+3,\frac{m}{2}}^X
    \end{align*}
vanishes. Monodromy invariance gives the existence statement directly. We give a second proof using Jun Li's degeneration formula. 
\end{abstract}
\textbf{\textit{Key words: }}Gromov--Witten invariant, complete intersections, monodromy group

\tableofcontents

\section{Introduction}

Let $X$ be a smooth projective variety over $\mbC$. Denote by 
\begin{align*}
    \left\langle  \gamma_1, \dots, \gamma_n\right\rangle_{0,n,\beta}^X
\end{align*}
the primary Gromov--Witten invariant counting genus 0 curves in $X$ of degree $\beta\in H_2(X,\mbZ)$ passing through cycles Poincaré dual to cohomology classes $\gamma_1,\dots,\gamma_n\in H^\star(X,\mbC)$. The virtual dimension of the moduli space of stable maps is $(\dim X-3)+\langle\beta,c_1(T_X)\rangle +n$.

Hu systematically applied monodromy invariance and WDVV equation to complete intersections \cite{Hu15}. Arg\"uz, Bousseau, Pandharipande, and Zvonkine apply degeneration and monodromy invariance to develop an algorithm for computing Gromov--Witten invariants of complete intersections \cite{ABPZ}. The case when $X$ is a complete intersection of two quadrics of even dimension $\dim X=m$ is special, since in this case the monodromy group is finite and the proposed algorithm has to be modified. The paper \cite{Hu21} systematically studies the genus 0 Gromov--Witten theory of $X$. Denote by $\varepsilon_1,\dots, \varepsilon_{m+3}$ the specified basis from \cite{Hu21} in the primitive cohomology $H^m_\prim(X,\mbQ)$. Hu shows \cite[Theorem 4.6]{Hu21} that the WDVV equation and monodromy invariance allow to determine all invariants from the 4-point invariants and the invariant
\begin{align}\label{main_correlator}
\left\langle\varepsilon_1,\dots,\varepsilon_{m+3}\right\rangle_{0,m+3,\frac{m}{2}}^X. 
\end{align}

The case $m=2$ corresponds to a degree  4 del Pezzo surface, a blow-up of $\mbP^2$ at 5 points in general positions (no three on a line).  Genus 0 Gromov--Witten theory of a blow-up of $\mbP^2$ was studied by Gottsche--Pandharipande in \cite{GP98} where the WDVV equation was used to compute Gromov--Witten invariants recursively. 

Hu computes the case $m=4$ \cite[Theorem 1.6]{Hu21}
\begin{align*}
\left\langle\varepsilon_1,\dots,\varepsilon_{7}\right\rangle_{0,7,2}^X = \frac{1}{2}
\end{align*}
and predicts \cite[Conjecture 1.5]{Hu21}
\begin{align}\label{conj}
\left\langle\varepsilon_1,\dots,\varepsilon_{m+3}\right\rangle_{0,m+3,\frac{m}{2}}^X = \frac{(-1)^{\frac{m}{2}}}{2}.
\end{align}
 
In the present paper we prove the following

\begin{theorem}\label{maintheorem}
There exists an orthogonal basis $e_1,\dots,e_{m+3}$ in $H^m_\prim(X,\mbQ)$ for which $$\left\langle e_1,\dots,e_{m+3}\right\rangle_{0,m+3,\frac{m}{2}}^X = 0.$$
\end{theorem}

\begin{remark}
    Our theorem does not determine Hu’s invariant but is compatible with the conjectured value. Indeed, let $A$ be the matrix sending a basis $\varepsilon_1,\dots,\varepsilon_{m+3}$ to a basis $e_1,\dots,e_{m+3}$
\begin{align*}
    e_i = \sum_{j=1}^{m+3}A_{ij}\varepsilon_j.
\end{align*}
Then
    \begin{align}\label{perm}
        \left\langle e_1,\dots,e_{m+3}\right\rangle_{0,m+3,\frac{m}{2}}^X=  
        \mathrm{perm}(A)\left\langle \varepsilon_1,\dots,\varepsilon_{m+3}\right\rangle_{0,m+3,\frac{m}{2}}^X, 
        \end{align}
        where
        \begin{align*}
            \mathrm{perm}(A)=\sum_{\sigma\in S_{m+3}}\prod_{i=1}^{m+3}  A_{i\sigma(i)}. 
    \end{align*}
This follows from the monodromy invariance of Gromov--Witten invariants, see Section \ref{monodromie}. By \cite[§5.2-5.3]{DK73} the algebraic monodromy group of $X$ is the Weyl group of Dynkin type $D_{m+3}$
\begin{align*}
    W(D_{m+3}) = S_{m+3}\ltimes (\mbZ/2\mbZ)^{m+2},
\end{align*}

generated by permuting coordinates in $\mbQ^{m+3}$ and changing signs at even number of places
\begin{align*}
(x_1,\dots,x_i,\dots, x_j,\dots, x_{m+3})\mapsto (x_1,\dots,-x_i,\dots, -x_j,\dots, x_{m+3}).
\end{align*}
Then the invariant $\left\langle \varepsilon_{i_1},\dots,\varepsilon_{i_{m+3}}\right\rangle_{0,m+3,\frac{m}{2}}^X$
 has a chance to be non-zero only if $(i_1,\dots,i_{m+3})$ is a permutation of $(1,\dots, m+3)$. Then (\ref{perm}) follows.

 By (\ref{perm}), any orthogonal basis obtained from $A$ with $\mathrm{perm}(A)=0$ proves our theorem. For instance, 
 \begin{align*}
 e_i=\varepsilon_i \quad i=1,\dots,m+1,\quad e_{m+2}= \varepsilon_{m+2}+\varepsilon_{m+3}, \quad e_{m+3} = \varepsilon_{m+2}-\varepsilon_{m+3}.   
\end{align*}
\end{remark}

We give a second proof of Theorem \ref{maintheorem} in Section \ref{degeneration} based on degeneration formula. We address the proof of the Conjecture \ref{conj} in a separate publication \cite{GZ}.

\subsection{Acknowledgements} \label{ackn}
I am grateful to Dimitry Zvonkine, for his suggestion to explore this problem, numerous discussions and patient guidance. I thank referees for careful reading of the earlier version of this paper and many helpful comments. The author is partially supported by International Laboratory of Cluster Geometry NRU HSE, RF Government grant, ag. № 075-15-2021-608 dated 08.06.2021

\textbf{AI disclosure}

All ideas are our own; AI was used only for editing.

\section{The degeneration family}\label{smooth_family}

Here we describe the degeneration we use to prove Theorem \ref{maintheorem}, first proposed to study Gromov--Witten theory by Maulik--Pandharipande in \cite[Section 0.5.4]{MP06}.

Let $X$ be a complete intersection of two quadric hypersurfaces in $\mbP^{m+2}$ given by zeros degree 2 polynomials 
\begin{align*}
    X = \{ f_1 = f_2 = 0\}, 
\end{align*}
whose differentials are linearly independent at each point of $X$, so that $X$ is smooth. Assume that $\{f_1=0\}$ and $\{f_2=0\}$ are smooth quadrics.
Consider the family 
\begin{align*}
   \widehat{\pi}: \widehat{X} \to \mbA^1, 
\end{align*}
given by $\widehat{X} = \{f_1 = tf_2 + g_1g_2 = 0\}\subset \mbA^1 \times \mbP^{m+2}$, where $g_1, g_2$ are generically chosen degree~1 polynomials and $t$ is the coordinate on the target~$\mbA^1$.
The generic fiber is homeomorphic to $X$ and the special fiber $\widehat{X}_0$ is the union of two quadrics 
\begin{align*}
    X_1  &= \{t = f_1 = g_1 = 0\},\\
    X_2  &= \{t = f_1 = g_2 = 0\}
\end{align*}
in $\{0 \} \times \mbP^{m+2}$ intersecting transversally along a smooth divisor $D$
\begin{align*}
    D = \{t =  f_1 = g_1 =  g_2 = 0\}.
\end{align*}
The variety $\widehat{X}$ is not smooth: its singular locus is
\begin{align*}
    Z = \{ x\in \widehat{X} | \; df_1, f_2dt + tdf_2 + g_1dg_2 + g_2dg_1 \text{ are linearly dependent at } x \}.
\end{align*}
Below in Lemma \ref{Zcutout} we prove that $Z$ is given by
\begin{align*}
    Z = \{t=f_1 = f_2 = g_1 = g_2 = 0\}.
\end{align*}
We collect the degrees and dimensions of the above varieties in the table: 
\begin{center}
\begin{tabular}{ |c|c|c| } 
 \hline
 & degree & dim \\ 
  
 $X$ & 2,2 & $m$ \\ 
 $X_1$ & 2,1 & $m$ \\ 
 $X_2$ & 2,1 & $m$ \\ 
 $D$ & 2,1,1 & $m-1$ \\ 
 $Z$ & 2,2,1,1 & $m-2$ \\ 
 \hline
\end{tabular}
\end{center}

Precomposing $\widehat{\pi}$ with a blow up map $\widetilde{\pi}: \widetilde{X}:=\Bl_{X_2}\widehat{X}\to \widehat{X}$, we obtain a proper family 
\begin{align*}
    p:= \widehat{\pi} \circ \widetilde{\pi} : \widetilde{X} \to \mbA^1.
\end{align*}

The full preimages of the components $X_1$, $X_2$ of the special fiber and of their intersection~$D$ under the blow up map are
\begin{align}\label{components}
    &\widetilde{\pi}^{-1}X_1 = X_1 \bigcup_Z \mbP(N_{X_2/Z}),\\
    &\widetilde{\pi}^{-1}D = D \bigcup_Z \mbP(N_{X_2/Z}),\\
    &\widetilde{X}_2 := \widetilde{\pi}^{-1}X_2 = \Bl_Z X_2,
\end{align}
where $N_{X_2/Z}=\mathcal{O}_Z(1)\oplus \mathcal{O}_Z(2)$ is the normal bundle to $Z$ in~$X_2$.

To apply Jun Li's degeneration formula, see Section \ref{degeneration}, we need the total space of the family to be smooth and the special fiber to be a union of two smooth varieties intersecting transversally along a smooth subvariety. And, indeed, we have the following lemma.

\begin{lemma}\label{totalspace}
    The total space of the proper family 
    \begin{align*}
        p: \widetilde{X} \to \mbA^1
    \end{align*}
    is smooth with smooth generic fiber $p^{-1}(t)$ deformation equivalent to $X$ and the special fiber $\widetilde{X}_0:= p^{-1}(0)$ isomorphic to $X_1 \bigcup_D \widetilde{X}_2$, where the intersection along~$D$ is transversal.
\end{lemma}
\begin{proof}
Because the degree~1 polynomials $g_1$ and $g_2$ are chosen generically, the varieties $X_1$ and $X_2$, cut out of $\mbP^{m+2}$
by the equations $f_1=g_1=0$ and $f_1=g_2=0$ respectively, are smooth and intersect transversally along $D: \{f_1 = g_1=g_2=0 \}$, which is also smooth. The blow-up of $X_2$ along the smooth subvariety $Z: \{f_1=f_2=g_1=g_2=0\}$ preserves the smoothness and the transversality of the intersection.

Now let us check the smoothness of the total space $\widetilde{X}$. The morphism 
\begin{align}\label{smooth_loci}
        \widetilde{X}\backslash \mbP(N_ZX_2) \to \widehat{X}\backslash Z,
   \end{align}
induced from the blow-up map $\widetilde{\pi}: \widetilde{X} \to \widehat{X}$ is an isomorphism, being the blow-up map of a smooth variety $\widehat{X}\backslash Z$ along a smooth divisor $X_2 \backslash Z$. We are left to check the smoothness of $\widetilde{X}$ over $Z$, i.e. at any point $\widetilde{z} \in \widetilde{X}$ lying over a point $z \in Z \subset \widehat{X}$.
Denote by $[t_0: t_1 : \dots : t_{m+2}]$ the homogeneous coordinates of $\mbP^{m+2}$. Without loss of generality, we can assume that $t_0 \not=0$ at the point~$z$. We then let $t_0=1$ so that $(t_1, \dots t_{m+2})$ are  coordinates on an affine chart of $\mbP^{m+2}$ containing~$z$. 

Recall that the equations of $\widehat{X}$ are $f_1 = 0$ and $ tf_2 + g_1 g_2 = 0$, where $t$ is the coordinate on~$\mbA^1$. The Weil divisor $X_2$ is cut out of $\widehat{X}$ by two additional equations $t=0$, $g_2=0$. By abuse of notation, we still denote by $f_1, f_2, g_1, g_2$ the polynomials in variables $t_1, \dots, t_{m+2}$ obtained by plugging $t_0=1$ into the homogeneous polynomials. Then an affine chart of $\widehat{X}$ containing $z$ is the spectrum of the algebra
$$
A = \mbC[t_1, \dots, t_{m+2}, t] / (f_1, tf_2+g_1g_2).
$$
The ideal $I \subset A$ of $X_2$ is generated as an $A$-module by two elements $T$ and $G_2$ that correspond to the equations $t=0$ and $g_2=0$ of $X_2$ and satisfy the relations
$$
tG_2-Tg_2=0, \quad Tf_2+g_1G_2=0,
$$
obtained by replacing $t$ and $g_2$ by $T$ and $G_2$ in the relations of~$A$ in all possible ways linear in $T$ and~$G_2$. Thus we obtain a presentation of the algebra  $\bigoplus_k I^k$:
$$
 \bigoplus_k I^k = \mbC[t_1, \dots, t_{m+2}, t, T,G_2] / (tG_2-Tg_2, Tf_2+g_1G_2,f_1, tf_2+g_1g_2),
$$
where $f_1, f_2, g_1, g_2$ are polynomials in $t_1, \dots, t_{m+2}$ of degrees $2, 2, 1, 1$.

The case $T(\widetilde{z}) = G_2(\widetilde{z})=0$ is not possible by the projective spectrum construction, used to define the blow-up. Then we consider two cases for $\widetilde{z}$.

{\bf 1. We have $T(\widetilde{z})  \not= 0$.} Then we can set $T=1$ to obtain an affine chart of $\widetilde{X}$ containing $\widetilde{z}$. Of the four relations defining $\bigoplus_k I^k$,
$$
tG_2-g_2, \quad f_2+g_1G_2, \quad f_1, \quad tf_2 + g_1g_2, 
$$
the last one becomes redundant as it follows from the first two. So we are left with three equations:
$$
f_1, \quad f_2 + g_1G_2, \quad tG_2-g_2
$$
in variables $t_1, \dots, t_{m+2}, t, G_2$. The matrix of partial derivatives reads
$$
\def\arraystretch{2.2}
\begin{array}{c||ccc}
& f_1& f_2 + g_1G_2& tG_2-g_2 \\
\hline
\hline
\frac{\partial}{\partial t_i} & \frac{\partial f_1}{\partial t_i} & \frac{\partial f_2}{\partial t_i} + G_2 \frac{\partial g_1}{\partial t_i}  & -\frac{\partial g_2}{\partial t_i} \\
\frac{\partial}{\partial G_2} & 0 & g_1 & t\\
\frac{\partial}{\partial t} & 0 & 0 & G_2\\
\end{array},
$$
where the first line actually denotes $m+2$ lines for $i$ from 1 to $m+2$. Now, by the smoothness of $X$ and the genericity assumption on $g_1$ and $g_2$, the locus $Z = \{f_1 = f_2 = g_1=g_2 = 0\}$ is itself smooth, in other words, the vectors $\partial f_1/\partial t_i$,  $\partial f_2/\partial t_i$, $\partial g_1/\partial t_i$, $\partial g_2/\partial t_i$ are linearly independent whenever $f_1 = f_2 = g_1=g_2 = 0$. Thus the columns of the matrix are linearly independent even without the last two rows, so that $\widetilde{X}$ is smooth at~$\widetilde{z}$ of dimension $m+4-3 = m+1$.

{\bf 2. We have $G_2(\widetilde{z})  \not= 0$.} Then we can set $G_2=1$ to obtain an affine chart of $\widetilde{X}$ containing $\widetilde{z}$. Of the four relations defining $\oplus_k I^k$,
$$
t-g_2T, \quad Tf_2+g_1, \quad f_1, \quad tf_2 + g_1g_2, 
$$
the last one becomes redundant as it follows from the first two. So we are left with three equations:
$$
f_1, \quad Tf_2 + g_1, \quad t-g_2T
$$
in variables $t_1, \dots, t_{m+2}, t, T$. The matrix of partial derivatives reads
$$
\def\arraystretch{2.2}
\begin{array}{c||ccc}
& f_1 & Tf_2 + g_1 & t-g_2T\\
\hline
\hline
\frac{\partial}{\partial t_i} & \frac{\partial f_1}{\partial t_i} & T  \frac{\partial f_2}{\partial t_i} + \frac{\partial g_1}{\partial t_i}  & -T \frac{\partial g_2}{\partial t_i} \\
\frac{\partial}{\partial T} & 0 & f_2& -g_2\\
\frac{\partial}{\partial t} & 0 & 0 & 1\\
\end{array},
$$
where the first line once again denotes $m+2$ lines for $i$ from 1 to $m+2$. Now, as in the previous case, the first two columns of the matrix are linearly independent even without the two last rows, because the vectors
$\partial f_1/\partial t_i$,  $\partial f_2/\partial t_i$, and $\partial g_1/\partial t_i$ are linearly independent whenever $f_1 = f_2 = g_1=g_2 = 0$. And the third column is linearly independent from the first two because it has a 1 in the last row. Thus the columns of the matrix are linearly independent, so that $\widetilde{X}$ is again smooth at~$\widetilde{z}$ of dimension $m+4-3 = m+1$.
\end{proof}

Not every fiber $p^{-1}(t)$ is smooth, but there is an open neighborhood $U\subset \mbA^1$ of $0$ such that the only singular fiber over it is $p^{-1}(0)$. 

\begin{lemma}\label{Zcutout}
Let $f_1, \cdots, f_k$ be homogeneous polynomials in $m+k+1$ variables. Assume that $\{f_1 = \dots = f_{k-1} = 0 \}$ and $\{f_1 = \dots  = f_k = 0 \}$ are  smooth complete intersections in $\mbP^{m+k}$. Consider the family 
\begin{align*}
    \widehat{X} = \{f_1 = \cdots = f_{k-1} = tf_k + g_1g_2 = 0\} \subset \mbP^{m+k}\times \mbA^1,
\end{align*}
where $g_1$ and $g_2$ are generically chosen polynomials of degrees $d_k^1, d_k^2$ such that $d_k^1 + d_k^2 = \deg f_k$.  Then  the singular locus $Z$ of the restriction of $\widehat{X}$ to a sufficiently small neighborhood $U$ of $t=0$ in $\mbA^1$ is given by 
  \begin{align*}
    Z = \{f_1 = \cdots = f_k = t = g_1 = g_2 = 0\}.
\end{align*}
\end{lemma}
\begin{proof}
    Denote by $t_0, \dots, t_{m+k}$ the homogeneous coordinates in $\mbP^{m+k}$ and let $F_k  = tf_k + g_1g_2$. The matrix of partial derivatives of $f_1, \dots, f_{k-1}, F_k$ with respect to $t_0, \dots, t_{m+k}, t$ reads
$$
\begin{bmatrix}
\dfrac{\partial f_1}{\partial t_0} & \dots & \dfrac{\partial f_1}{\partial t_{m+k}} & 0 \\
\vdots &   & \vdots & \vdots \\
\dfrac{\partial f_{k-1}}{\partial t_0} & \dots & \dfrac{\partial f_{k-1}}{\partial t_{m+k}} & 0 \\
 \quad \dfrac{\partial F_k}{\partial t_0} \quad  & \quad \dots \quad & 
 \quad \dfrac{\partial F_k}{\partial t_{m+k}} \quad
 &\quad f_k \quad
\end{bmatrix},
$$
where 
$$
\frac{\partial F_k}{\partial t_i} = t \frac{\partial f_k}{\partial t_i} + g_1 \frac{\partial g_2}{\partial t_i} + g_2 \frac{\partial g_1}{\partial t_i}.
$$
A point $z \in \widehat{X}$ is smooth iff the evaluation of this matrix at $z$ has maximal rank (that is, rank~$k$). First note that if $f_k=t=g_1=g_2=0$ at $z$, the last row of the matrix vanishes, so that the rank of the matrix is at most $k-1$ and hence $z$ belongs to the singular locus as claimed in the statement of the lemma. Now we must prove that all the other points of $\widehat{X}$ over an appropriately chosen open set $U \subset \mbA^1$ are smooth.

First we specify the choice of the open set $U$ and the meaning of genericity of $g_1$ and $g_2$. The polynomials $g_1$ and $g_2$ must be chosen in such a way that $\{ f_1 = \dots = f_{k-1} = g_1=0 \}$ and $\{ f_1 = \dots = f_{k-1} = g_2=0 \}$ are smooth complete intersections. By Bertini's theorem, this is a generic choice. Similarly, the open set $U$ is the set of values $t$ for which $\{f_1 = \dots = f_{k-1} = F_k = 0 \}$ is a smooth complete intersection. Let us prove that this open set is nonempty.

Consider the set $S$ of polynomials $F$ of degree $d_k = \deg f_k$ in $t_0, \dots, t_{m+k}$ such that the variety $\{ f_1 = \dots = f_{k-1} = F = 0 \}$ is a smooth complete intersection. By Bertini's theorem, this set is open and dense in the space of all degree $d_k$ polynomials. The family $t f_k + g_1g_2$ extends to a map from $\mbP^1$ to the projectivization of the space of all degree $d_k$ polynomials. Moreover, the image of $t=\infty$ lies in the projectivization of $S$, because by assumption $X$ is a smooth complete intersection. If follows that all but a finite number of points of $\mbP^1$ have their images in the projectivization of $S$. The neighborhood $U$ in the lemma is obtained by removing all points  $t \in \mbA^1$ whose image is not in $S$, except the point $t=0$. In the restriction of the family $\widehat{X}$ to $U$, all fibers are smooth except the fiber $X_0$ over $t=0$. 

Statement 1: if $z$ is a singular point of $\widehat{X}$, then $f_k(z) = 0$. Proof: the block $\partial f_i/ \partial t_j$ has maximal rank $k-1$ at every point, since by assumption $\{f_1 = \dots = f_{k-1} = 0\}$ is a smooth complete intersection. If $f_k(z) \not=0$, the last row of the matrix is linearly independent from the other ones, so the matrix has rank~$k$.

Statement 2: if $z$ is a smooth point of the fiber $X_t$ of $\widehat{X}$ for $t= t(z)$ then $z$ is also a smooth point of $\widehat{X}$. Proof: 
the matrix of partial derivatives of the equations of $X_t$ is the same matrix as above, but without the last column. If $X_t$ is smooth at $z$, this submatrix already has rank~$k$. 

Now, by construction of $U$, the only singular fiber lies over $t=0$. Further, by the choice of $g_1$ and $g_2$, the irreducible components of this fiber $\{ f_1 = \dots = f_{k-1} = g_1 = 0 \}$ and $\{ f_1 = \dots = f_{k-1} = g_2 = 0 \}$ are both smooth complete intersections, so that the singular points of $X_0$ necessarily lie on their intersection, where $t=g_1=g_2 = 0$. We conclude that if $z$ is a singular point of $\widehat{X}$, we have $f_k(z) = t(z) = g_1(z) = g_2(z) =0$. 

\end{proof}

\section{The cohomology of $X$}
Here we recall some background on the cohomology of complete intersections in $\mathbb{P}^N$, with particular emphasis on the cohomology of even-dimensional complete intersections of two quadrics. Good sources are the books \cite{V02}, \cite{GH78}.

\subsection{Cohomology of complex smooth projective varieties}
Let $X$ be a connected smooth projective variety over $\mathbb{C}$ of complex dimension $m$ in an ambient projective space  $\mbP^N$. Denote by $\omega\in H^2(X,\mbQ)$ the restriction to $X$ of the cohomology class Poincaré dual to a hyperplane in  $\mbP^N$.

The Lefschetz operator acts on the singular cohomology by cupping with $\omega$
\begin{align*}
     \omega : H^{k}(X, \mbQ) \to H^{k+2}(X, \mbQ).
\end{align*}

Lefschetz hyperplane theorem states that 
the restriction to generic hyperplane section $X_H$
\begin{align*}
    H^k(X, \mbQ)\to H^k(X_H,\mbQ)
\end{align*}
is an isomorphism for $k<m-1$ and is injective for $k=m-1$. If $X\subset \mbP^{m+1}$ is a hypersurface of degree $d$, it can be viewed as a hyperplane section inside $\nu_d(\mbP^{m+1})$  via degree $d$ Veronese embedding $\mbP^{m+1}\xrightarrow{\nu_d}\mbP^{\binom{m+1+d}{d}-1}$. Now by Lefschetz hyperplane theorem $H^k(\nu_d(\mbP^{m+1}),\mbQ)\to H^k(X,\mbQ)$ is an isomorphism for $k<m$ and is injective for $k=m$.

Hard Lefschetz theorem states that 
\begin{theorem}
    For any $0\leq k\leq m$
    \begin{align*}
        \omega^{m-k}: H^{k}(X,\mbQ)\to H^{2m-k}(X,\mbQ)
    \end{align*}
    is an isomorphism.
\end{theorem}

One of formal consequences of the Hard Lefschetz theorem is the Lefschetz decomposition. 

\begin{definition}
     \label{Def:PrimitiveCoh}
For $0\leq k\leq m$ the \textit{primitive cohomology} of degree $k$ is defined by
    $$H^k_{\prim}(X, \mbQ) := \ker (\omega^{m-k+1}: H^k(X,\mbQ)\to H^{2m-k+2}(X,\mbQ))$$
    \end{definition}
For instance, for $k=0,1$ all cohomology is primitive.

The Lefschets decomposition is the following theorem
\begin{theorem}
\label{Th:Lefdecomp}
For all $0\leq k \leq m$
    \begin{align*}
        H^k(X,\mbQ)\simeq \bigoplus_{\substack{r\geq 0\\0\leq k-2r\leq m}} \omega^{r}H^{k-2r}_{\prim}(X,\mbQ).
    \end{align*}
\end{theorem}

Below we will focus on complete intersections in projective space.
\subsection{Cohomology of smooth complete intersections in $\mbP^N$}

Let $X$ be a smooth hypersurface in $\mbP^N$ of dimension $m$ and degree $d$. Denote by $\omega$ the restriction to $X$ of the cohomology class Poincaré dual to a hyperplane in  $\mbP^N$. 

\begin{remark}
    According to the general definition, $H^0(X,\mbQ)$ is also part of the primitive cohomology of~$X$. However, by abuse of language, when talking about the primitive cohomology of complete intersections we will only refer to the degree~$m$ part.
\end{remark}

\begin{proposition} \label{Prop:AmbientCoh} 
The following is true
    \begin{itemize}
        \item If $i$ is odd and $i\not= m$, we have $H^i(X,\mbZ) = 0$.
        \item   If $i$ is even and $i< m$, we have $H^i(X,\mbZ) = \mbZ \left< \omega^{i/2} \right>$.
        \item   If $i$ is even and $i> m$, we have $H^i(X,\mbZ) = \mbZ \left< \frac{\omega^{i/2}}d \right>$.
    \end{itemize}
\end{proposition}

\begin{proof}

    By Lefschetz hyperplane theorem, for $k< m$ the restriction maps
    \begin{align*}
        H^k(\mbP^N,\mbZ)\to H^k(X,\mbZ)
    \end{align*}
    are isomorphisms and 
    \begin{align*}
        H^m(\mbP^N,\mbZ)\to H^m(X,\mbZ)
    \end{align*}
    is injective. 
    Recall the cohomology of projecive space $H^{2k}(\mbP^N, \mbZ)$ are generated by $k$-th power of the hyperplane class, and are zero elsewhere. The second assertion follows. For the last one, use the Lefschetz hyperplane theorem for homology. It states that for $k< m$ the push-forward maps
    \begin{align*}
        H_k(X,\mbZ)\to H_k(\mbP^N,\mbZ)
    \end{align*}
    are isomorphisms. That is $H_{2k}(X,\mbZ)\simeq \mbZ$, $2k<m$, in particular are torsion free. By Poincaré duality $H_{2k}(X,\mbZ)\simeq H^{2m-2k}(X,\mbZ)$.

    The Poincaré pairing 
    \begin{align*}
        \langle,\rangle: H^{2m-2k}(X,\mbQ)\times H^{2k}(X,\mbQ)\to \mbQ
    \end{align*}
    is non-degenerate, and we know already that for $k<m$ even $H^{k}(X,\mbZ)\simeq \mbZ\langle\omega^{\frac{k}{2}}\rangle$, and are zero for odd $k$ since homology of projective space are. Then $H^{2m-2k}(X,\mbZ) \simeq \mbZ\langle c\omega^{m-k}\rangle$ for some $c\in \mbQ$. Then $\langle c\omega^{m-k},\omega^k\rangle=cd\in \mbZ$, so $c\in \frac{1}{d}\mbZ$. 

    The first and the last assertions follow.
\end{proof}

\begin{definition}
\label{Def:ambient}
We call {\em ambient} the cohomology classes proportional to the powers of $\omega$.
\end{definition}
    
According to the previous proposition, all cohomology classes of $X$ are ambient except possibly in degree~$m = \dim X$. Denote by $H^m_{\prim}(X,\mbQ) \subset H^m(X,\mbQ)$ the subspace of primitive cohomology.

If $m$ is odd, then $H^m_{\prim}(X,\mbQ) = H^m(X,\mbQ)$, because $H^{m+2}(X,\mbQ) = 0$. If $m$ is even, we have the following proposition.

\begin{proposition} \label{Prop:PrimitiveCoh}
    Suppose $m = \dim X$ is even. Then we have 
    $$
    H^m(X,\mbQ) = \mbQ \left< \omega^{\frac{m}{2}} \right>\oplus H^m_{\rm prim}(X,\mbQ).
    $$
Moreover, the direct summands $\mbQ\left< \omega^{\frac{m}{2}} \right>$ and $H^m_{\rm prim}(X,\mbQ)$ are orthogonal with respect to the Poincaré pairing. 
\end{proposition}
\begin{proof}
    This is the content of Lefschetz decomposition in Theorem \ref{Th:Lefdecomp} where the only non-zero summands are for $r=0$ and $r=\frac{m}{2}$.
\end{proof}
For the next two propositions assume $m$ is even.
\begin{proposition}
    The group $H^m(X,\mbZ)$ is torsion-free; in particular,  the natural map $H^m(X,\mbZ)\to H^m(X,\mbQ)$ is an embedding.
\end{proposition}
\begin{proof}
    It follows from the universal coefficient theorem, stating that the following sequence is exact:
\begin{align*}
    0 \to \Ext^1(H_{m-1}(X, \mbZ), \mbZ) \to H^m(X, \mbZ) \to \Hom (H_m(X, \mbZ), \mbZ) \to 0, 
\end{align*}
 where $H_{m-1}(X, \mbZ) = H_{m-1}(\mbP^{N}, \mbZ) = 0$ by  the Lefschetz hyperplane theorem, and the rightmost group is torsion free being a group of homomorphisms to a torsion free group $\mbZ$.
\end{proof}

\begin{proposition}
    \label{Prop:IntPrimCoh}
    The element $\omega^{\frac{m}{2}} \in H^m(X,\mbZ)$ is not a multiple of any other integral cohomology class. Denote by $H^m_{\rm prim}(X,\mbZ) = H^m_{\rm prim}(X,\mbQ) \cap H^m(X,\mbZ)$. Then the lattice
    $$
\mbZ \left< \omega^{\frac{m}{2}} \right> \oplus H^m_{\prim}(X,\mbZ)
    $$
    has index $d$ in $H^m(X,\mbZ)$. In particular, for $d>1$, the integer cohomology $H^m(X,\mbZ)$ does {\em not} decompose into a direct sum of ambient and primitive abelian subgroups.
    \end{proposition}
\begin{proof}
    The first assertion is immediate due to embedding $H^m(X,\mbZ)\to H^m(X,\mbQ)$ and Proposition \ref{Prop:PrimitiveCoh}. For the  second statement, let $x\in H^m(X,\mbZ)$. Its image in $H^m(X,\mbQ)$ has the form $a\omega^{\frac{m}{2}}+x_{\prim}$ for some $a\in \mbQ$.  The Poincaré pairing on $H^m(X,\mbZ)$ is integer valued, then $\langle x, \omega^{\frac{m}{2}}\rangle = ad \in \mbZ$, that is $a\in \frac{1}{d}\mbZ$. On the other hand,  the image of the lattice $\mbZ \left< \omega^{\frac{m}{2}} \right> \oplus H^m_{\prim}(X,\mbZ)$ in  $H^m(X,\mbQ)$ consists of
    integer linear combinations of $\omega^{\frac{m}{2}}$ and primitive classes. We conclude that the group $H^m(X,\mbZ)$ is the (non-split) extension 
    \begin{align*}
     1\to\mbZ \left< \omega^{\frac{m}{2}} \right> \oplus H^m_{\prim}(X,\mbZ)\to H^m(X,\mbZ)\to \mbZ/d\mbZ\to1    
    \end{align*}
 
\end{proof}

\begin{corollary}
\label{Prop:NoTorsion} 
The graded abelian group $H^\star(X,\mbZ)$ has no torsion; in particular, the natural map $H^\star(X,\mbZ) \to H^\star(X,\mbQ)$ is injective.
\end{corollary}

The above generalizes to smooth complete intersections $X$ of dimension $m$ given by $k$ homogeneous equations $f_1=0,\dots, f_k=0$ of degrees $d_1,\dots, d_k$. 
One can think of $X$ as an iterated hyperplane section of the line bundles $\nu_{d_i}^* \mathcal{O}(1)$ where 
\begin{align*}
 \nu_{d_i}: \mbP^{m+2} \to \mbP^{\binom{m+2+ d_i}{d_i} -1}   
\end{align*}
is a degree $d_i$ Veronese embedding.

\begin{proposition} 
The following is true
    \begin{itemize}
        \item If $i$ is odd and $i\not= m$, we have $H^i(X,\mbZ) = 0$.
        \item   If $i$ is even and $i< m$, we have $H^i(X,\mbZ) = \mbZ \left< \omega^{i/2} \right>$.
        \item   If $i$ is even and $i> m$, we have $H^i(X,\mbZ) = \mbZ \left< \frac{\omega^{i/2}}{d_1\dots d_k} \right>$.
        \item  if $m$ is odd $H^m_{\prim}(X,\mbZ)=H^m(X,\mbZ)$; if $m$ is even the lattice  $\mbZ \left< \omega^{\frac{m}{2}} \right> \oplus H^m_{\prim}(X,\mbZ)$ has index $d_1\dots d_k$ in $H^m(X,\mbZ)$.
        \item The graded abelian group $H^m(X,\mbZ)$ has no torsion.
    \end{itemize}
\end{proposition}

\subsection{Cohomology of complete intersection of two quadrics}

Fix an even $m \geq 2$ and let $X$ be 
a complete intersection of two quadrics in $\mbP^{m+2}$. From above we know its middle cohomology decompose $H^m(X,\mbQ)=\mbQ\langle \omega^{\frac{m}{2}}\rangle\oplus H^{m}_{\prim}(X,\mbQ)$. First, we find the Betti number
\begin{lemma}\label{betti}
    $\dim H^{m}_{\prim}(X,\mbQ) = m+3$
\end{lemma}
\begin{proof}
    From the structure of cohomology of complete intersections, the topological Euler characteristic of $X$ is
    \begin{align*}
        \chi(X) = m+1 + \dim H^m_\prim(X, \mbQ).
    \end{align*}

Hence the lemma follows if we can show that $\chi(X)=2m+4$. We will do so by determining the Euler class of $X$.

The conormal sequence of a closed embedding $X \to \mbP^{m+2}$ with ideal sheaf $I$
\begin{align*}
    0\to I/I^2 \to \Omega_{\mbP^{m+2}}\vert_X \to \Omega_X \to 0
\end{align*}
is exact since $X$ is smooth.

 The ideal sheaf $I$ is generated locally by $f_1$ and $f_2$ and locally $I/I^2$ splits into direct sum of conormal line bundles of two quadrics  
\begin{align*}
    I/I^2 = (f_1,f_2)/(f_1^2,f_1f_2,f_2^2) = (f_1)/(f_1^2)\oplus(f_2)/(f_2^2).
\end{align*}
The natural map $I/I^2 \to \mathcal{O}(-2)\vert_{X} \oplus \mathcal{O}(-2)\vert_{X}$ is an isomorphism of vector bundles.
Dually, the normal bundle to $X$ inside $\mbP^{m+2}$ is $\mathcal{O}(2)\vert_{X} \oplus \mathcal{O}(2)\vert_{X}$ and fits into a short exact sequence 
\begin{align*}
    0 \to T_X \to T_{\mbP^{m+2}}\vert_{X} \to \mathcal{O}(2)\vert_{X} \oplus \mathcal{O}(2)\vert_{X} \to 0.
\end{align*}
We get 
\begin{align*}
    c(T_X) = \dfrac{c(T_{\mbP^{m+2}}\vert_X)}{c(\mathcal{O}(2)\vert_X)^2}\\
    = \dfrac{(1+\omega)^{m+3}}{(1+2\omega)^2}.
\end{align*}

We compute
\begin{align*}
    \chi(X) = \int_Xc_m(T_X) = \int_{\mbP^{m+2}}4\omega^2c_m(T_X) = 4\Coef_{\omega^{m+2}}\omega^2\dfrac{(1+\omega)^{m+3}}{(1+2\omega)^2}.
\end{align*}
As a result of a direct computation we have
\begin{align*}
    \chi(X) = 2m+4,
\end{align*}
and, hence
    \begin{align*}
        \dim H^m_\prim(X, \mbQ) = m+3.
    \end{align*}
\end{proof}

\textbf{Integral primitive lattice} 

Here we explain how the lattice $H^m_\prim(X,\mbZ) = H^m_\prim(X, \mbQ)\cap H^m(X, \mbZ)$ of integral primitive elements is identified with the root lattice of Dynkin type $D_{m+3}$.

By Reid's rational basis theorem \cite[Lemma~3.13]{Reid},
if $X\subset\mathbb P^{m+2}$ is a smooth intersection of two quadrics,
with \(m\ge4\) even, then
\[
 \omega^{\frac{m}{2}},\quad \zeta_0,\ldots,\zeta_{m+2}
\]
form a basis of \(H^m(X,\mbQ)\).
Here \(\omega\) is the hyperplane class, \(\zeta=[S]\) is a maximal
plane class, and \(\zeta_j=[S_j]\), where \(S_j\) is obtained from \(S\)
by changing the sign of the \(j\)-th coordinate in diagonal coordinates
for the two quadrics.

We identify the torsion-free group \(H^m(X,\mbZ)\) with its image in
\(H^m(X,\mbQ)\), and define

\begin{align*}
    L_{\racine} = \sum_{i=1}^{m+3} \mbZ \alpha_i,
\end{align*}

where
\begin{align}\label{root}
 \alpha_i=\zeta_{i-1}-\zeta_i\quad 1\le i<m+3,\qquad
 \alpha_{m+3}=2\zeta-\omega^{\frac{m}{2}}+\zeta_{m+1}+\zeta_{m+2}.   
\end{align}

\begin{proposition}
The inclusion \(L_{\racine}\subseteq H^m_{\prim}(X,\mbZ)\)
is an equality, hence an isomorphism of lattices with their
intersection pairings.
\end{proposition}
\begin{proof}
The plane intersection formulas of Reid \cite[Lemma~3.10]{Reid},
as recalled in \cite[\S7]{Hu21}, give
\begin{align*}
 &\langle\zeta_i,\zeta_j\rangle=(-1)^{\frac m2}
 \begin{cases}
 \left\lfloor\frac m4+1\right\rfloor,&i=j,\\
 \left\lfloor\frac m4\right\rfloor,&i\ne j,
 \end{cases}\\
 &\langle\zeta,\zeta\rangle=(-1)^{\frac m2}\left\lfloor\frac m4+1\right\rfloor,\\
 & \langle\zeta,\zeta_j\rangle=(-1)^{\frac m2-1}\left\lfloor\frac{m+2}{4}\right\rfloor,\\
 &\langle\omega^{\frac m2},\omega^{\frac m2}\rangle=4,\\
 & \langle\zeta,\omega^{\frac m2}\rangle=\langle\zeta_j,\omega^{\frac m2}\rangle=1.
\end{align*}
Thus every \(\alpha_i\) is integral and orthogonal to \(\omega^{\frac{m}{2}}\), so
\(L_{\mathrm{root}}\subseteq H^m_{\prim}(X,\mbZ)\) and we are left to show the opposite inclusion.

Choose integral test classes
\[
 \gamma_j=\zeta_{j-1} \quad1\le j<m+3,\qquad \gamma_{m+3}=\zeta .
\]
The displayed intersection formulas imply
\[
 P:=\bigl(\langle\gamma_i,\alpha_j\rangle\bigr)_{1\leq i,j\leq m+3}
 =(-1)^{\frac{m}{2}}\begin{pmatrix}T&v\\0&1\end{pmatrix},
\]
\[
 T=\begin{pmatrix}
 1&0&\cdots&0&0\\
 -1&1&\ddots&\vdots&\vdots\\
 0&-1&\ddots&0&\vdots\\
 \vdots&\ddots&\ddots&1&0\\
 0&\cdots&0&-1&1
 \end{pmatrix},\qquad
 v=(-1,\ldots,-1,0)^{\mathsf T},
\]
where \(T\) is an \((m+2)\times(m+2)\) lower triangular matrix
with diagonal entries \(1\). Consequently,
\[
 \det P=(-1)^{\frac{m(m+3)}2}=\pm1.
\]
In particular, the roots are linearly independent. There are $m+3=\dim H^m_\prim(X, \mbQ) $ of them by Lemma \ref{betti}, hence they form
a basis of \(H^m_{\prim}(X,\mbQ)\).
For \(x\in H^m_{\prim}(X,\mbZ)\), write
\(x=\sum_j\beta_j\alpha_j\) with \(\beta_j\in\mbQ\) for its image in $H^m_{\prim}(X,\mbQ)$.
Since \(x\) and the \(\gamma_i\) are integral, then $\langle \gamma_i,x\rangle = \sum_j \langle \gamma_i, \alpha_j\rangle\beta_j\in \mbZ$. We may write
\[
 u:=\bigl(\langle \gamma_i,x\rangle\bigr)_{i=1}^{m+3}\in\mbZ^{m+3},\qquad u=P\beta.
\]
As \(P^{-1}\) has integral entries, \(\beta=P^{-1}u\in\mbZ^{m+3}\).
Thus \(x\in L_{\mathrm{root}}\), proving equality.
\end{proof}

\begin{lemma}\label{Xcoh}
   Let $X$ be an smooth complete intersection of two quadrics in $\mbP^{m+2}$ of even dimension. Then 
\begin{align}
    H^k(X,\mbZ) = \begin{cases}
        \mbZ \left< \omega^{\frac{k}{2}} \right>, & 0 \leq k \leq m-2, \;\; k \text{ even}\\
        \mbZ \left< \omega^{\frac{k}{2}} / 4 \right>, & m+2 \leq k \leq 2m, \;\; k \text{ even}\\
        \mbZ \left< \zeta  \right> \oplus H^{m}_\prim(X, \mbZ), \hspace{1cm} & k=m , \\
        0, & \text{otherwise}.
    \end{cases}
\end{align}
Moreover, we have an identification 
\begin{align*}
    (H^m_{\prim}(X,\mbZ), \langle,\rangle) = L_\racine = (R(D_{m+3}), (-1)^{\frac{m}{2}}(,))
\end{align*}
with the root lattice of Dynkin type $D_{m+3}$.
\end{lemma}

\begin{proof}
We are left to show the degree $k=m$. Let $x$ be an integral class, then $x-\langle x,\omega^{\frac{m}{2}}\rangle \zeta$ is primitive and we have the displayed splitting. Now we explain the isomorphism with the root lattice. Let \(e_i\) be the
standard coordinate vectors of \(\mbQ^{m+3}\), with $(,)$ the standard scalar product. The root system
\[
 \Phi(D_{m+3})=\{\pm e_i\pm e_j:1\le i<j\le m+3\}
\]
has simple roots \(e_i-e_{i+1}\) for \(i<m+3\) and
\(e_{m+2}+e_{m+3}\). The simple roots generate the root lattice
\begin{align*}
    R(D_{m+3}) = \left\{(a_1,\ldots,a_{m+3})\in\mbZ^{m+3}:
          \sum_{i=1}^{m+3} a_i\in2\mbZ\right\}.
\end{align*}

Set \(\varepsilon_i=\zeta+\zeta_{i-1}-\omega^{\frac{m}{2}}/2\) for \(1\le i\le m+3\).
The intersection formulas above yield
\[
 \langle\varepsilon_i,\omega^{\frac{m}{2}}\rangle=0,\qquad
 \langle\varepsilon_i,\varepsilon_j\rangle=(-1)^{\frac{m}{2}}\delta_{ij},
\]
and, moreover, by (\ref{root}) we have
\begin{align*}
 &\alpha_i=\varepsilon_i-\varepsilon_{i+1},~ 1\leq i<m+3,\\
 &\alpha_{m+3}=\varepsilon_{m+2}+\varepsilon_{m+3}.
\end{align*}

Thus the map \(\mbQ^{m+3}\to H^m_{\prim}(X,\mbQ),~  e_i\mapsto\varepsilon_i\) sends the simple roots exactly to the classes $\alpha_i$ and sends the scalar product $(-1)^{\frac{m}{2}}(,)$ to the pairing $\langle,\rangle$.

\end{proof}

\begin{remark}
 Our goal is to show there exists a basis $\omega^{\frac{m}{2}}, e_1,\dots, e_{m+3}$ in $H^m(X,\mbQ)$ with the property that only $e_{m+2}$ restricts to $X_1$ non-trivially after lifting to $\widetilde{X}$, see Corollary \ref{n1Argument}. It turns out that the degeneration formula with insertions $e_1, \dots, e_{m+3}$ simplifies drastically and, moreover, the relative Gromov--Witten invariants of $(X_1,D)$ constituting the formula, vanish by simple dimension count, see Section \ref{degeneration}. 
\end{remark}

\begin{corollary}\label{Zprim}
    Since $Z$ is also a smooth complete intersection of two quadrics but in $\mbP^m$, we obtain

    $$
    H^{m-2}_\prim(Z, \mbZ) \simeq \mbZ^{m+1}
    $$

    \begin{align*}
    H^k(Z,\mbZ) = \begin{cases}
        \mbZ \left< \omega_Z^{\frac{k}{2}} \right>, & 0\leq k \leq m-4 , k \text{ even}\\
        \mbZ \left< \omega_Z^{\frac{k}{2}} /4 \right>, & m\leq k \leq 2(m-2) , k \text{ even}\\
        \mbZ \oplus H^{m-2}_\prim(Z, \mbZ) , & k = m-2 \\
        0, & \text{otherwise}
    \end{cases}
    \end{align*}
\end{corollary}

We compute the cohomology of quadrics $X_1, X_2, D$, defined in the beginning of Section 2. We use Propositions \ref{Prop:AmbientCoh}, \ref{Prop:IntPrimCoh} and  
short exact sequences 
    \begin{align*}
        &0 \to T_{X_i} \to T_{\mbP^{m+1}}\vert_{X_i} \to \mathcal{O}(2)\vert_{X_i} \to 0,\\
        &0 \to T_{D} \to T_{\mbP^{m}}\vert_{D} \to \mathcal{O}(2)\vert_{D} \to 0.
    \end{align*} 
\begin{lemma}\label{X1X2Dcoh}
The only non-zero cohomology groups of $X_i$ are
    $$
    H^{m}_\prim(X_i, \mbZ) \simeq \mbZ
    $$
\begin{align*}
    H^k(X_i,\mbZ) = \begin{cases}
        \mbZ\left< \omega_{X_i}^{\frac{k}{2}} \right>, & 0\leq k \leq m-2 , k \text{ even}\\
        \mbZ\left< \omega_{X_i}^{\frac{k}{2}}/2 \right>, & m+2\leq k \leq 2m , k \text{ even};
       \end{cases}
    \end{align*}
    and the sublattice $\mbZ\langle\omega_{X_i}^{\frac{m}{2}}\rangle\oplus H^m_{\prim}(X_i,\mbZ) \subset H^m(X_i,\mbZ)$ is of index 2; 
    
    The only non-zero cohomology groups of $D$ are
    \begin{align*}
       H^k(D,\mbZ) = \begin{cases}
        \mbZ\left< \omega_{D}^{\frac{k}{2}} \right>, & 0\leq k \leq m-2 , k \text{ even}\\
        \mbZ\left< \omega_{D}^{\frac{k}{2}}/2 \right>, & m\leq k \leq 2m-2 , k \text{ even}.
       \end{cases}
\end{align*}
In particular, $$
    H^{m-1}_\prim(D, \mbZ) = 0.
    $$
\end{lemma}

Write $m=2r$. Here we comment about middle cohomology of a quadric $Q\subset \mbP^{2r+1}$, for us $Q$ is $X_1$ or $X_2$. Denote by $\omega_Q$ the restriction to $Q$ of the hyperplane class. It is known one can choose coordinates such that $Q$ is given by equation $x_0y_0+\dots+x_ry_r=0$. The linear subspace $\{y_1=\dots = y_r = 0\}\simeq \mbP^{r+1}$. Its intersection with $Q$ is the union of two $r$-planes
\begin{align*}
    &A: x_0=y_1=\dots=y_{r}=0,\\
    &B: y_0=y_1=\dots=y_r=0.
\end{align*}
The hyperplanes have complex codimensions $r$ inside $Q$ hence their Poincaré dual classes belong to $H^{2r}(Q,\mbZ)=H^m(Q,\mbZ)$, denote them by $a$ and $b$. We have an equation $a+b=\omega_{Q}^{\frac{m}{2}}$. We have 
\begin{lemma}
In the ordered basis
$(a,b)$ the intersection matrix is
\[
  \begin{pmatrix}1&0\\0&1\end{pmatrix}
  \quad \text{for }m\equiv 0\pmod 4,\qquad
  \begin{pmatrix}0&1\\1&0\end{pmatrix}
  \quad \text{for }m\equiv 2\pmod 4.
\]
    
    \end{lemma}
\begin{proof}

     The class $a$ is represented by an $r$-plane $A\simeq \mbP^r\subset Q$ and $\omega_Q$ restricts to a hyperplane class on $A$, then
    \begin{align*}
        \int_Q a \omega_Q^r = \int_A (\omega_Q\vert_A)^r = 1.
    \end{align*}

Since $\omega_Q^r = a+b$ we have 
    \begin{align*}
        \int_Q(a+b)a = \int_Q(a+b)b=1.
    \end{align*}
Then the intersection matrix in the basis $(a,b)$ has the form
\begin{align*}
    \begin{pmatrix}
        u&1-u\\1-u&u,
    \end{pmatrix}
\end{align*}
where $u:=\int_Qa^2$. The normal sequence for $A\subset Q\subset\mathbb P^{2r+1}$ is
\[
  0\longrightarrow N_{A/Q}\longrightarrow
  \mathcal O_A(1)^{\oplus(r+1)}\longrightarrow\mathcal O_A(2)
  \longrightarrow 0.
\]
By the self-intersection formula,
\[
  u=\int_A c_r(N_{A/Q})
   =\Coef_{\omega^r}\frac{(1+\omega)^{r+1}}{1+2\omega}
   =\sum_{k=0}^r\binom{r+1}{k}(-2)^{r-k}
   =\frac{1+(-1)^r}{2}.
\]
Then $u=0$ for $r$ odd and $u=1$ for $r$ even.
\end{proof}

We have 
\begin{align*}
    H^m(Q, \mbZ) = \mbZ\langle a\rangle\oplus \mbZ\langle b\rangle
\end{align*}
and a sublattice of index 2 $\mbZ\langle a+b\rangle \oplus \mbZ \langle a-b\rangle$. We have $\omega_Q^{\frac{m}{2}}(a-b)=(a+b)(a-b)=0$ for both parities of $r$, hence the class $a-b$ is primitive, of square $(a-b)^2=2u-2(1-u)=2(2u-1)=2(-1)^r$.

Denote the generator of $H^m_{\prim}(X_1, \mbZ)$ by $\beta$ and the generator of $H^m_{\prim}(X_2, \mbZ)$ by $\theta$ and the generators of $H^{m-2}_\prim(Z, \mbQ)$ by $Z_1,\dots, Z_{m+1}$.
We will treat $Z_i$ as analogs of $\varepsilon_i$ on $Z$, and choose them to satisfy $$\langle Z_i,Z_j\rangle = (-1)^{\frac{m-2}{2}}\delta_{ij}.$$ From the comment above we see $$\langle\beta,\beta\rangle = \langle\theta,\theta\rangle=2(-1)^r.$$
 
For convenience, we use the shorthands:
\begin{align*}
    \Hcal_1&:=\omega_{X_1}^{\frac{m}{2}},\\
    \Hcal_2&:=\pi^*\omega_{X_2}^{\frac{m}{2}},\\
    \Hcal_Z&:=j_*\pi_E^* \omega_Z^{\frac{m}{2}-1},\\
    \Zcal_i&:=j_*\pi_E^* Z_i, \hspace{2em} 1\leq i\leq m+1\\
    \Theta&:=\pi^*\theta,
\end{align*}

where the maps involved are from a cartesian diagram of a  blow-up 
\begin{center}
\begin{tikzcd}
	E & {\widetilde{X_2}} \\
	Z & {X_2}
	\arrow["j"', hook, from=1-1, to=1-2]
	\arrow["{\pi_E}", from=1-1, to=2-1]
	\arrow["\pi", from=1-2, to=2-2]
	\arrow["i", hook, from=2-1, to=2-2]
\end{tikzcd},
\end{center}
where $E = \mbP(N_{X_2/Z})$ is an exceptional divisor.

For needs of Section \ref{section_restriction} we need the middle cohomology of the component $\widetilde{X}_2$ (\ref{components}) of the special fiber of $p: \widetilde{X}\to \mbA^1$. 

\begin{lemma}\label{blowup_HS}
The map 
\begin{align*}
     \pi^* + j_*\pi_E^*: H^m(X_2, \mbZ)\oplus H^{m-2}(Z, \mbZ)(-1) &\to H^m(\widetilde{X}_2, \mbZ)\\
(\omega^{\frac{m}{2}}_{X_2}, 0)&\mapsto \Hcal_2\\
(\theta,0)&\mapsto\Theta\\
(0, \omega^{\frac{m}{2}-1}_Z)&\mapsto\Hcal_Z\\
(0, Z_i) &\mapsto\Zcal_i, \hspace{2em} 1\leq i\leq m+1.
\end{align*}

is an isomorphism of $\mbZ$-Hodge structures.
\end{lemma}

See \cite[Theorem 7.31]{V02} for a proof.

\section{Monodromy of complete intersections}\label{monodromie}
In this section we show the small monodromy group of the family $p:\widetilde{X}\to \mbA^1$ from Section \ref{smooth_family}, restricted to sufficiently small analytic disk around 0 is trivial. Then from local invariant cycle theorem it will follow the restriction map 
\begin{align*}
    H^\star(\widetilde{X},\mbQ)\to H^\star(X,\mbQ)
\end{align*}
is surjective, see Lemma \ref{surj}
. For completeness we review the description of algebraic monodromy groups of general smooth  complete intersections.  

\subsection{Monodromy invariance}

We first recall the property of Gromov--Witten invariants of being invariant under monodromy. Let $\mathcal{X}\xrightarrow{p}U$ be a smooth projective family with fiber $X$ over $u\in U$.  Given a path $a(t),~ a(0)=u$ in $U$, denote by $\gamma_i(t),\beta(t)$  the parallel transport of classes $\gamma_i\in H^\star(X,\mbC)$ and $\beta\in H_2(X,\mbZ)$ along the path $a(t)$. The \textit{deformation invariance} of Gromov--Witten invariants means that the function
\begin{align*}
    \left\langle  \gamma_1(t), \dots, \gamma_n(t)\right\rangle_{0,n,\beta(t)}^{X_{a(t)}}
\end{align*}
is constant. 

Let the path $a(t)$ be a loop at $u$ and denote by $g_a$ its image under the monodromy representation 
\begin{align*}
    \pi_1(U,u)\to \Aut(H^\star(X,\mbC)),
\end{align*}
determined by the local systems $R^ip_*\mbC$. Denote by $g_a\cdot \gamma_i,~ g_a\cdot \beta$ the images under operator $g_a$. The \textit{monodromy invariance} of Gromov--Witten invariants means
\begin{align*}
    \left\langle  g_a\cdot\gamma_1, \dots, g_a\cdot\gamma_n\right\rangle_{0,n,g_a\cdot\beta}^X=\left\langle  \gamma_1, \dots, \gamma_n\right\rangle_{0,n,\beta}^X 
\end{align*}
and follows from the deformation invariance. 

\subsection{Small monodromy group}

Recall the proper family $p: \widetilde{X}\to \mbA^1$ with a smooth total space from Section \ref{smooth_family}, restricted to sufficiently small analytic disk $\Delta$ such that $p$ is singular only over $0$. Restricted further to the complement of a special fiber, it is a smooth proper map $p^{\prime}: \widetilde{X}^{\prime}\to \Delta^*$ between smooth varieties, with fiber $X$  

\begin{center}
\begin{tikzcd}
  \widetilde{X}^{\prime} \arrow[r] \arrow[d, "p^{\prime}"]
    &  \widetilde{X} \arrow[d, "p"] \\
  \Delta^* \arrow[r]
& \Delta \end{tikzcd}.
\end{center}

We call the \textit{small monodromy group} of the family $p^\prime$ the image of the action of $\pi_1(\Delta^*,t)=\mbZ$ in the complex algebraic group $\Aut(H^m_\prim(X,\mbQ)\otimes \mbC)$. Denote it by $\Gamma$.

We show
\begin{lemma}
The small monodromy group $\Gamma$ is trivial.    
\end{lemma}
 \begin{proof}
      The triviality follows from the following observations. First, the monodromy action
      preserves the integer classes $\im (H^m_\prim(X, \mbZ)\to H^m_\prim(X, \mbZ)\otimes \mbC)$.
Choosing a basis in $H^m(X, \mbZ)$, we can identify $H^m(X, \mbZ)$ with a standard lattice $\Lambda$ and hence the group $\Gamma$ is a subgroup of a discrete group $GL(\Lambda)$. 

Second, this action preserves the non-degenerate symmetric bilinear form on $H^{m}_{\prim}(X, \mbQ)$, the intersection pairing, since it is the same at each fiber. Then $\Gamma$ sits inside the orthogonal group $O(H^{m}_{\prim}(X, \mbQ))$, which is dense inside $O(H^{m}_{\prim}(X, \mbR))$(in the analytic topology).

Intersection pairing on $H^{m}_{\prim}(X, \mbR)$ is definite of sign $(-1)^{\frac{m}{2}}$, hence the group $O(H^{m}_{\prim}(X, \mbR))$ is compact. 

Being discrete and compact, the monodromy group $\Gamma$ is finite.

Moreover, Clemens proved \cite[Theorem 7.36]{Cle77} that the elements $g$ of $\Gamma$ are of the form $g = e^{A_g}$, where $A_g$ are nilpotent complex-valued matrices. On the other hand, from the finiteness of $\Gamma$, each $g$ is of finite order $e^{n_gA_g} = 1$, $n_g \in \mathbb{N}$ which implies $e^{A_g} = 1$. Indeed, since the sum of nilpotent commuting matrices is nilpotent, 
\begin{align*}
   e^{n_gA_g} =  (1+ \widetilde{A}_g)^{n_g} = 1
\end{align*}
for some nilpotent $\widetilde{A}_g$. The last equality can only happen if $\widetilde{A}_g = 0$. Then, 
\begin{align*}
e^{A_g} = 1 + \widetilde{A}_g = 1.
\end{align*}
\end{proof}

\begin{corollary}\label{inv}
    All the classes in $H^{m}_{\prim}(X, \mbQ)$ are  $\Gamma$-invariant. 
\end{corollary}

\subsection{Algebraic monodromy of complete intersections}

For completeness we recall the algebraic monodromy groups of general complete intersections. When $X$ is a complete intersection there is a particular family with fiber $X$, the universal one. Namely, let $$U\subset \prod_i\mbP(H^0(\mbP^{m+k},\mathcal{O}_{\mbP^{m+k}}(d_i)))$$ be an open subset selecting smooth complete intersections of multi-degree $(d_1, \dots, d_k)$ in $\mbP^{m+k}$.
Fix $u\in U$ and denote by $\pi_1(U,u)$ the fundamental group of $U$ based at $u$. Denote by $X$ the corresponding complete intersection.

By Lefschetz hyperplane theorem the restriction map $$H^i(\mbP^{m+k},\mbQ)\to H^i(X,\mbQ)$$ is isomorphism for $i\neq m, ~1\leq i\leq 2m$. By Lefschetz decomposition theorem, $$H^m(X,\mbQ)=H^m(\mbP^{m+k},\mbQ)\oplus H^m_{\prim}(X,\mbQ).$$

Consider the universal family of smooth complete intersections of multi-degree $(d_1, \dots, d_k)$ in $\mbP^{m+k}$
\[\begin{tikzcd}
	{\mathcal{X}} & {U\times \mbP^{m+k}} \\
	& U
	\arrow[hook, from=1-1, to=1-2]
	\arrow["p"', from=1-1, to=2-2]
	\arrow[from=1-2, to=2-2]
\end{tikzcd}\]
induced by projection to $U$ and formed by such collections $(f_1, \dots, f_k,x)$ that $f_1(x)=0,~\dots,~ f_k(x)=0$. The family is smooth and proper and the sheaf $R^m p_* \mbQ$ is a local system determined by the action of $\pi_1(U,u)$ by linear automorphisms on its fiber $H^m(X,\mbQ)$ over $u$
\begin{align*}
    \pi_1(U,u)\to \Aut(H^m(X,\mbQ)).
\end{align*}
 The \textit{algebraic monodromy group} $G$ of $X$ is, by definition, the Zariski
closure of the image of $\pi_1(U,u)$ in the complex algebraic group  $\Aut(H^m(X,\mbQ)\otimes\mbC)$.

Suppose $X$ is not of multi-degrees $(3)$ nor $(2,2)$.
If the dimension $m$ of $X$
is odd, then the intersection pairing defines a skew-symmetric non-degenerate bilinear form $\eta$ on $H^m(X,\mbQ)$.  Moreover, as the odd cohomology of $\mbP^{m+k}$
is zero,
the entire cohomology $H^m(X,\mbQ)$ is primitive. Let $\mathrm{Sp(V)}$ be the complex symplectic
group of automorphisms of
$V:=H^m_{\prim}(X,\mbQ)\otimes \mbC$
endowed with $\eta$. As the monodromy preserves $\eta$, we necessarily have $G \subseteq \mathrm{Sp}(V)$. 
We have 
\begin{theorem}[\cite{ABPZ}, Proposition 4.1] 
If $\dim X = m$ is 
 odd, then the monodromy group $G$ acting
on the primitive cohomology $V$ is as large as possible: $G=\mathrm{Sp}(V).$
\end{theorem}

Suppose again that $X$ is not of multi-degrees $(3)$ nor $(2,2)$. When $m$ is even, the intersection pairing  defines a symmetric non-degenerate bilinear form $\eta$ on $H^m(X,\mbQ)$. The primitive cohomology $H^m_\prim(X,\mbQ)$ is orthogonal wrt $\eta$ to the image of the restriction map $H^m(\mbP^{m+k},\mbQ)\to H^m(X,\mbQ)$. Denote by the same $\eta$ the restriction to  $H^m_\prim(X,\mbQ)$ of the symmetric non-degenerate bilinear form. Let $O(V)$ be the group of orthogonal transformations of $V:=H^m_\prim(X,\mbQ)\otimes  \mbC$.
As the monodromy group preserves $\eta$, we necessarily have $G \subseteq O(V)$. 

We have 
\begin{theorem}[\cite{ABPZ}, Proposition 4.2; \cite{DK73}, XIX, §5.2-5.3] 
If $\dim X = m$ is 
 even, then the monodromy group $G$ acting
on the primitive cohomology $V$ is as large as possible: $G=O(V)$, except if $X$ is
a cubic surface or a complete intersection of two quadrics. Furthermore,
\begin{itemize}
    \item [$(i)$] if $X$ is a cubic surface, then $G$ = $W(E_6)$,
\item [$(ii)$] if $X$ is a complete intersection of two quadrics, then $G$ = $W(D_{m+3})$,
\end{itemize}
where $W(R)$ denotes the Weyl group of the root system $R$. 
\end{theorem}

\begin{remark}
    When $X$ is a complete intersection of two quadrics, we showed in Lemma \ref{Xcoh} that
\begin{align*}
 H^m_{\prim}(X,\mbZ) = L_\racine = R(D_{m+3})    
\end{align*}
as lattices with scalar product. The Weyl group acts on $R(D_{m+3})$ and hence on $H^m_{\prim}(X,\mbZ)$. 
\end{remark}

\section{The restriction map}\label{section_restriction}

Recall a proper family with smooth total space $\widetilde{X}\xrightarrow{p} \Delta$ from Section \ref{smooth_family}.
The embedding of generic fiber $X$ into $\widetilde{X}$ induces the 
restriction map
\begin{align}\label{restriction}
    H^{\star}(\widetilde{X}, \mbQ) \to H^{\star}(X, \mbQ).
\end{align}

\begin{lemma}\label{surj}
    The restriction map (\ref{restriction}) is surjective
\end{lemma}
\begin{proof}
This follows from Corollary \ref{inv} and the local invariant cycle theorem \cite[Theorem~11.43]{PS08} that says that after restricting the family $\widetilde{X}\xrightarrow{p}\mbA^1$ to a sufficiently small disc $\Delta$ around $0$, we have
\[
\operatorname{Im}\bigl(
H^m(\widetilde X,\mbQ)\longrightarrow H^m(X,\mbQ)
\bigr)
=
\ker(T-\mathrm{id}),
\]
where $T$ is the local monodromy operator.
By Corollary~\ref{inv}, $T$ acts trivially on
$H^m_{\prim}(X,\mbQ)$; it also fixes $\omega^{\frac{m}{2}}$.
Hence the restriction map is surjective in degree $m$.
In every other degree, surjectivity follows because the
cohomology of $X$ is generated by powers of the hyperplane
class, which extends over the total space.
\end{proof}

From Lemma \ref{Xcoh} we already know that the rank of $H^m(X, \mbQ)$ is $m+4$. Denote its basis by
\begin{align*}
    \omega_X^{\frac{m}{2}}, e_1, \dots, e_{m+3},
\end{align*}
where $(e_i)_{1\leq i \leq m+3}$ form a basis in $H^m_{\prim}(X,\mbQ)$. 

In this section we construct a basis in cohomology group $H^m(\widetilde{X}, \mbQ)$ and compute the matrix of the restriction map (\ref{restriction})
in the chosen basis to further apply it in Section \ref{degeneration}.



 By \cite[Theorem 5.7]{Cle77}, the total space $\widetilde{X}$  retracts onto its special fiber $\widetilde{X}_0$. In particular, the induced map
\begin{align*}
    H^{\star}(\widetilde{X}_0,\mbQ)\to H^{\star}(\widetilde{X}, \mbQ)
\end{align*}
is an isomorphism of $\mbZ$-graded $\mbQ$-algebras.

Now we proceed to construct a basis of $H^m(\widetilde{X}_0, \mbQ)$. Till the end of this section we assume coefficients to be $\mbQ$ unless otherwise stated.

\begin{lemma}\label{Xhat0_basis}
The group $H^m(\widetilde{X}_0, \mbQ)$ is the kernel of restriction map
\begin{align*}
    &H^m(X_1)\oplus H^m(\widetilde{X}_2)\xrightarrow {\gamma}H^m(D)\\
    &(a,b)\mapsto a|_D - b|_D.
\end{align*}
  The classes  $(\Hcal_1, \Hcal_2), (\Hcal_1, \Hcal_Z - \Hcal_2), (\beta, 0), (0, \Theta), (0, \Zcal_i)$,  $i=1,\dots,m+1$ form a basis in $H^m(\widetilde{X}_0)$. 
\end{lemma}

\begin{proof}
The Mayer-Vietoris long exact sequence applied to space $\widetilde{X}_0$ gives
    \begin{align*} 
         \cdots\rightarrow H^{m-1}(D)\rightarrow 
         H^m(\widetilde{X}_0)\rightarrow H^m(X_1)\oplus H^m(\widetilde{X}_2)\xrightarrow {\gamma}H^m(D)=\mbQ\langle \omega^{\frac{m}{2}}_D \rangle \rightarrow \cdots
    \end{align*}
    The first group $H^{m-1}(D)$ is $0$ since $D$ has no odd cohomology. The map $\gamma : (a,b) \mapsto a|_D - b|_D$ is surjective and $H^m(\widetilde{X}_0)$ is its kernel. 
    
    From Lemma \ref{X1X2Dcoh} and Lemma \ref{blowup_HS} tensored with $\mbQ$:
    \begin{align*}
     \pi^* + j_*\pi_E^*: H^m(X_2, \mbQ)\oplus H^{m-2}(Z, \mbQ)(-1) &\simeq H^m(\widetilde{X}_2, \mbQ),
     \end{align*}
    we know the classes $(\Hcal_1, 0), (\beta, 0), (0, \Hcal_2), (0, \Theta), (0, \Hcal_Z), (0, \Zcal_i)$ form a basis of the group $H^m(X_1)\oplus H^m(\widetilde{X}_2)$. Further we compute their restrictions to $D$ in the following way(we abuse the notation by writing $\beta$ instead of $(\beta, 0)$, for instance):
    \begin{center}
    \begin{tabular}{ cc } 
    $\Hcal_1\mapsto \omega^{\frac{m}{2}}_D$& $\Hcal_2\mapsto \omega^{\frac{m}{2}}_D$ \\ 
    $\beta\mapsto0$ & $\Theta\mapsto0$  \\ 
    & $\Zcal_i\mapsto0$  \\
    & $\Hcal_Z\mapsto2\omega^{\frac{m}{2}}_D$.
    \end{tabular}
    \end{center}

   Indeed, by Lefschetz hyperplane theorem, we have $H^2(X_1)=H^2(D)=H^2(X_2)$, so $(\omega_{X_1})|_D=\omega_D=(\omega_{X_2})|_D$.  

   Note that $\beta|_D = c \omega^{\frac{m}{2}}_D$ for some constant $c$.
   The class $\beta$ is primitive, so $\beta \omega_{X_1}=0$. Hence, $\beta|_D (\omega_{X_1})|_D = c \omega^{\frac{m}{2}+1}_D = 0$ and so $c=0$.

Denote by $\widetilde{D}\subset \widetilde{X}_2$ the strict transform of $D$. We have isomorphisms $\widetilde{D}\simeq D$ and $E\cap \widetilde{D}\simeq Z$.
    Next we use a commutative diagram with left-lower and upper squares
cartesian.
\begin{center}
\[\begin{tikzcd}
	& Z & {\widetilde{D}} \\
	E & {\widetilde{X}_2} & D \\
	Z & {X_2}
	\arrow["{j^{\prime}}", hook, from=1-2, to=1-3]
	\arrow["{s^{\prime}}"', hook, from=1-2, to=2-1]
	\arrow["s"', hook, from=1-3, to=2-2]
	\arrow[no head, from=1-3, to=2-3]
	\arrow[shift left, no head, from=1-3, to=2-3]
	\arrow["j"', hook, from=2-1, to=2-2]
	\arrow["{\pi_E}", from=2-1, to=3-1]
	\arrow["\pi", from=2-2, to=3-2]
	\arrow[hook, from=2-3, to=3-2]
	\arrow["i", hook, from=3-1, to=3-2]
\end{tikzcd}\]
\end{center}

The class $\theta$ is primitive, so $\theta \omega_{X_2}=0$ and one similarly shows that $\Theta|_D=0$.

By functoriality of pullbacks 
   $\Hcal_2|_D = s^*\pi^*\omega^{\frac{m}{2}}_{X_2} = (\omega^{\frac{m}{2}}_{X_2})|_D = \omega^{\frac{m}{2}}_D$.

We are left to show that $\Zcal_i$ restricts to 0 and $\Hcal_Z$ restricts to $2\omega^{\frac{m}{2}}_D$. First we notice that for any cohomology class $z\in H^{m-2}(Z, \mbQ)$
\begin{align*}
    s^{*}j_*\pi_E^*z \in H^m(D,\mbQ) = \mbQ\langle \omega_D^{\frac{m}{2}}\rangle,
\end{align*}
    so $s^{*}j_*\pi_E^*z = c\omega_D^{\frac{m}{2}}$ for some constant $c$.

To compute the constant $c$, we integrate with the hyperplane class in an appropriate power. Applying for $z = Z_i,~ i=1, \dots, m+1$, we get 
\begin{align*}
    \Zcal_i|_D = j^{\prime}_*Z_i = c_i \omega_D^{\frac{m}{2}}\\
    \int_D \omega_D^{\frac{m}{2}-1} \cdot j^{\prime}_*Z_i = \int_Z \omega_Z^{\frac{m}{2}-1}Z_i = 0,
\end{align*}
since $Z_i,~i = 1, \dots, m+1$ are primitive. On the other hand, 
\begin{align*}
    \int_D \omega_D^{\frac{m}{2}-1} \cdot j^{\prime}_*Z_i = c_i\int_D \omega_D^{m-1} = 2c_i,
\end{align*}
hence $c_i = 0,~i = 1, \dots, m+1$.

Applying for $z = \omega_Z^{\frac{m}{2}-1}$,  we get
\begin{align*}
    \Hcal_Z|_D = j^{\prime}_* \omega_Z^{\frac{m}{2}-1} = c \omega_D^{\frac{m}{2}}\\
    \int_D \omega_D^{\frac{m}{2}-1} \cdot j^{\prime}_* \omega_Z^{\frac{m}{2}-1} = \int_Z \omega_Z^{m-2} = 4,
\end{align*}
since $Z$ is a complete intersection of two quadrics.

On the other hand,
\begin{align*}
    \int_D \omega_D^{\frac{m}{2}-1} \cdot j^{\prime}_* \omega_Z^{\frac{m}{2}-1} = c \int_D \omega_D^{m-1} = 2c,
\end{align*}
hence $c_Z = 2$.

We conclude that the generators of groups $H^m(X_1), H^m(\widetilde{X}_2)$ restrict to $D$ in the following way:

\begin{center}
\begin{tabular}{ cc } 
 $\Hcal_1\mapsto \omega^{\frac{m}{2}}_D$& $\Hcal_2\mapsto \omega^{\frac{m}{2}}_D$ \\ 
 $\beta\mapsto0$ & $\Theta\mapsto0$  \\ 
 & $\Zcal_i\mapsto0$  \\
 & $\Hcal_Z\mapsto2\omega^{\frac{m}{2}}_D$.
 \end{tabular}
\end{center}

From this one sees that the kernel of map $\gamma$ is generated by classes in the lemma.
\end{proof}

Combining Lemma \ref{Xhat0_basis} and
Lemma \ref{blowup_HS} tensored by $\mbQ$, we find the space $$H^m(\widetilde{X}_0)\subset H^m(X_1)\oplus H^m(X_2)\oplus H^{m-2}(Z)$$ has the following basis 

\begin{align*}
    &(\omega^{\frac{m}{2}}_{X_1}, \omega^{\frac{m}{2}}_{X_2}, 0),\\
    &(\omega^{\frac{m}{2}}_{X_1}, - \omega^{\frac{m}{2}}_{X_2}, \omega^{\frac{m}{2}-1}_Z),\\
    &(\beta, 0, 0),\\
    &(0, \theta, 0),\\
    &(0, 0, Z_i) \hspace{1em} \text{for } i = 1,\dots, m+1.
\end{align*}
\begin{lemma}
    The multiplication 
\begin{align*}
     H^m(\widetilde{X}_0) \otimes H^m(\widetilde{X}_0) \to H^{2m}(\widetilde{X}_0)
\end{align*}
is computed by the formula
\begin{align*}
    (a,b,c)(a^{\prime},b^{\prime},c^{\prime}) = (aa^{\prime}, bb^{\prime} - i_*(cc^{\prime}), 0).
\end{align*}
\end{lemma}
\begin{proof}
First of all we notice that the multiplication splits into the multiplication on $X_1$ and on $\widetilde{X}_2$. Indeed, let $a\in H^m(X_1)$ and $\Gamma \in H^m(\widetilde{X}_2)$, then $a\vert_D \cdot \Gamma \vert_D = (a\Gamma)\vert_D \in H^{2m}(D) = 0$, so the multiplication splits
\begin{align*}
    (a,\Gamma) \otimes (a^{\prime},\Gamma^{\prime}) \mapsto (aa^{\prime}, \Gamma\Gamma^{\prime}).
\end{align*}

    So, the claim in Lemma amounts to prove that under the multiplication 
    \begin{align*}
        H^m(\widetilde{X}_2) \otimes H^m(\widetilde{X}_2) \to H^{2m}(\widetilde{X}_2)
    \end{align*}
    one gets
    \begin{align*}
        (\pi^*b + j_*\pi_E^*c)(\pi^*b^{\prime} + j_*\pi_E^*c^{\prime}) = \pi^*(bb^{\prime}) - \pi^*i_*(cc^{\prime}).
    \end{align*}
    Indeed, to show this we use the 
isomorphism 
\begin{align*}
  \deg \circ p_{\widetilde{X}_2,*}: H^{2m}(\widetilde{X}_2)\to H^0(*) = \mbQ
\end{align*}
to argue the vanishing of
\begin{align*}
    \int_{\widetilde{X}_2} \gamma := \deg (p_{\widetilde{X}_2,*} \gamma) 
\end{align*}
implies the vanishing of $\gamma$. Here $\gamma \in H^{2m}(\widetilde{X}_2)$ and $p_{\widetilde{X}_2}$ is a proper map to a point. 

We compute
\begin{align*}
     \int_{\widetilde{X}_2}(\pi^*b + j_*\pi_E^*c)\cdot(\pi^*b^{\prime} + j_*\pi_E^*c^{\prime}) &= \\
      &=\int_{\widetilde{X}_2}(\pi^*(bb^{\prime}) + \pi^*b \cdot j_*\pi_E^*c^{\prime} + \pi^*b^{\prime} \cdot j_*\pi_E^*c + j_*\pi_E^*c \cdot j_*\pi_E^*c^{\prime}).
\end{align*}
The second summand is 
\begin{align*}
     \int_{\widetilde{X}_2}\pi^*b \cdot j_*\pi_E^*c^{\prime} = 
     \int_{X_2}b \cdot \pi_*j_*\pi_E^*c^{\prime} = 0
\end{align*}
since $\pi_*j_*\pi_E^*c^{\prime} = i_*\pi_{E,*}\pi_E^*c^{\prime}$ and we have $\pi_{E,*}\pi_E^*c^{\prime}=0$ from integration along the fibers of $\pi_E$. The same way the third summand is zero. Finally, we use excess intersection formula and projection formula to compute 
\begin{align*}
    \int_{\widetilde{X}_2} j_*\pi_E^*c \cdot j_*\pi_E^*c^{\prime} = \\
    \int_{E} j^*j_*\pi_E^*c \cdot \pi_E^*c^{\prime} = \\
    \int_{E} c_1(N_{\widetilde{X}_2/E}) \pi_E^*(cc^{\prime}) = \\
    - \int_E \xi \pi_E^*(cc^{\prime}) = \\
    -\int_Zcc^\prime,
    \end{align*}
where we used $N_{\widetilde{X}_2/E}=\mathcal{O}_E(-1),~ \xi := c_1(\mathcal{O}_E(1)),~ \pi_{E,*}\xi=1$.

\end{proof}

Now we proceed to compute the matrix of the restriction map (\ref{restriction}) in a basis found in Lemma \ref{Xhat0_basis}. We will essentially use the following
\begin{lemma}
    For all $\alpha \in H^{\star}(\widetilde{X},\mbQ)$ one has
    \begin{align}\label{key}
      \int_X\alpha|_{X}=\int_{X_1}\alpha|_{X_1}+\int_{\widetilde{X}_2}
      \alpha|_{\widetilde{X}_2} .
    \end{align}
\end{lemma}
\begin{proof}
Restrict the family $p:\widetilde{X}\to \mbA^1$ to a sufficiently small disk $\Delta$ around 0. Denote by $\kappa_a\in H^2_c(\Delta, \mbQ)$ the class of a point $a\in \Delta$. Since $\Delta$ is connected, $\kappa_t=\kappa_0$. The class $p^*\kappa_t$ is Poincaré dual to the funndamental class of the smooth fiber $X$, and $p^*\kappa_0$ is Poincaré dual to the sum of fundamental classes of the components $X_1, ~\widetilde{X}_2$ of the reduced 0-fiber, with coefficients 1. We have  
\begin{align*}
    \int_X \alpha\vert_X = \\ \int_{\widetilde{X}}\alpha \cdot p^*\kappa_t = \\\int_{\widetilde{X}}\alpha \cdot p^*\kappa_0 =\\ \int_X\alpha\vert_{X_1} + \int_X \alpha\vert_{\widetilde{X}_2}.  
\end{align*}

\end{proof}

\begin{corollary}\label{PPforXhat0}
    For any $\alpha = (a,b,c)(a^{\prime},b^{\prime},c^{\prime}) \in H^{2m}(\widetilde{X}_0, \mbQ)$
    \begin{align*}
      \int_X \alpha|_X &= \int_{X_1} \alpha|_{X_1} + \int_{\widetilde{X}_2}\alpha|_{\widetilde{X}_2} \\
      &=\int_{X_1}aa^{\prime} + \int_{X_2}bb^{\prime} - \int_{Z}cc^{\prime}.
    \end{align*}

    In particular, the intersection pairing on $H^m(\widetilde{X}_0)$ in the basis, found in Lemma \ref{Xhat0_basis}, is given by the matrix

\[
\bordermatrix{
  & (\Hcal_1,\Hcal_2) & (\Hcal_1,\Hcal_Z-\Hcal_2)
  & (\beta,0) & (0,\Theta) & (0,\Zcal_i) \cr
  (\Hcal_1,\Hcal_2)           & 4 & 0 & 0 & 0 & 0           \cr
  (\Hcal_1,\Hcal_Z-\Hcal_2)   & 0 & 0 & 0 & 0 & 0           \cr
  (\beta,0)                   & 0 & 0 & 2(-1)^{\frac{m}{2}} & 0 & 0           \cr
  (0,\Theta)                  & 0 & 0 & 0 & 2(-1)^{\frac{m}{2}} & 0           \cr
  (0,\Zcal_j)                 & 0 & 0 & 0 & 0 & (-1)^{\frac{m}{2}}\delta_{ij}\cr
}
\]

\end{corollary}

We arrive at the central proposition of this section
\begin{proposition}\label{restr}
    One can choose classes $e_i \in H^{m}_{\prim}(X, \mbQ), 1\leq i\leq m+3$ to be orthogonal with respect to intersection pairing, such that the restriction map (\ref{restriction}) 
    \begin{align*}
        H^{m}(\widetilde{X}, \mbQ) \to H^{m}(X, \mbQ)
    \end{align*}
     in such basis has the following diagonal matrix
    \begin{align*}
    &(\Hcal_1, \Hcal_2)\mapsto \omega_X^{\frac{m}{2}},\\ 
    &(\Hcal_1, \Hcal_Z - \Hcal_2)\mapsto 0,\\ 
    &(0, \Zcal_1) \mapsto e_1,\\
     &\cdots\\
    &(0, \Zcal_{m+1}) \mapsto e_{m+1},\\
    &(\beta, 0)\mapsto e_{m+2},\\
    &(0, \Theta) \mapsto e_{m+3}.
    \end{align*}
\end{proposition}
\begin{proof}
    By construction of the family $\widehat{\pi}: \widehat{X}\to \mathbb{A}^1$, the generic  fiber $X$ and the components $X_1$, $X_2$ of the special fiber are embedded into projective space $\mbP^{m+2}$. Denote by $\omega_{\mbP}\in H^2(\mbP^{m+2})$ the hyperplane class and its restrictions to $X_1, X_2, X$ by $\omega_{X_1}, \omega_{X_2}, \omega_{X}$. We compute

    \begin{align*}
        (\Hcal_1,\Hcal_2) = 
        (\omega_{X_1}^{\frac{m}{2}}, \pi^*\omega_{X_2}^{\frac{m}{2}}) = (\omega_{\mbP}|_{X_1},\pi^*\omega_{\mbP}|_{X_2})^{\frac{m}{2}} \mapsto 
        \omega_{\mbP}|_X^{\frac{m}{2}} = 
        \omega_{X}^{\frac{m}{2}},
    \end{align*}
    where we used the isomorphism $H^m(\widetilde{X}_0) \to H^m(\widetilde{X})$, induced by retraction $\widetilde{X}\to \widetilde{X}_0$.

From the Corollary \ref{PPforXhat0}, the image of the class $(\Hcal_1, \Hcal_Z - \Hcal_2)$ pairs trivially with every class in \(H^m(X,\mathbb Q)\), by the displayed pairing matrix and surjectivity of restriction. Nondegeneracy of the Poincaré pairing therefore implies that its image is zero.

    Then we use the freedom in the choice of basis $\omega_X^{\frac{m}{2}}, e_1,\dots, e_{m+3}$ of $H^m(X, \mbQ)$. We will require that  
   the intersection pairing 
    \begin{align*}
        H^m(X,\mbQ) \otimes H^m(X,\mbQ) \to H^{2m}(X, \mbQ) \simeq \mbQ
    \end{align*} in this basis
    has the form
    \begin{align*}
        &\langle \omega^{\frac{m}{2}},\omega^{\frac{m}{2}}\rangle = 4 ,\\
        &\langle \omega^{\frac{m}{2}},e_i\rangle = 0,  & 1\leq i\leq m+3\\
        &\langle e_i, e_j\rangle = (-1)^{\frac{m}{2}}\delta_{ij}, & 1\leq i,j\leq m+1, \\
        &\langle e_i, e_{m+2}\rangle = \langle e_i,  e_{m+3}\rangle = 0, & 1\leq i\leq m+1,\\
        &\langle e_{m+2},  e_{m+2}\rangle = \langle e_{m+3},  e_{m+3}\rangle = 2(-1)^{\frac{m}{2}} ,\\
        &\langle e_{m+2},  e_{m+3}\rangle = 0.
    \end{align*}

    For that we choose $e_i=\mathrm{res}(0,\mathcal{Z}_i) ~1\leq i\leq m+1,~ e_{m+2}=\mathrm{res}(\beta,0),~ e_{m+3}=\mathrm{res}(0,\Theta).$ Their Gram matrix is nondegenerate, so they are linearly independent. They are orthogonal to \(\omega^{m/2}\), hence primitive; their number is \(m+3\), so they form a basis of primitive cohomology.
\end{proof}

We call a basis $e_i$ with this property \textit{orthogonal with respect to intersection pairing}.
Clearly, under the proposed assignment intersection pairings on $H^m(\widetilde{X}_0)$ and $H^m(X)$ are compatible.  Choose the following explicit lifts of the classes $e_i$ under  (\ref{restriction})
    \begin{align*}
        \widetilde{e}_i = (0,\mathcal{Z}_i)\quad 1\leq i\leq m+1,\\
        \widetilde{e}_{m+2} = (\beta,0),~ \widetilde{e}_{m+3}=(0,\Theta).
    \end{align*}

\begin{corollary}\label{n1Argument}
     The lifts $\widetilde{e}_i$ have the following restrictions to $X_1$  
     \begin{align*}
         \widetilde{e}_i|_{X_1} = 
         \begin{cases}
         \beta \hspace{1cm} i = m+2,\\
         0 \hspace{1cm} \text{otherwise}
         \end{cases}
     \end{align*}
     and to $\widetilde{X}_2$
    \begin{align*}
        &\widetilde{e}_1|_{\widetilde{X}_2} =  \Zcal_1, \\
        &\cdots \\
        &\widetilde{e}_{m+1}|_{\widetilde{X}_2} =  \Zcal_{m+1},\\
        &\widetilde{e}_{m+2}|_{\widetilde{X}_2} = 0, \\
        &\widetilde{e}_{m+3}|_{\widetilde{X}_2} = \Theta.
    \end{align*}
\end{corollary}


\section{Degeneration formula}\label{degeneration}

The degeneration formula in numerical form expresses Gromov--Witten invariants of generic fiber $X$ in a family $p: \widetilde{X} \to \Delta$ with smooth total space, see Lemma \ref{totalspace},  with insertions in the image of  restriction map (\ref{restriction}) with coefficients in $\mbC$

\begin{align}
    H^{\star}(\widetilde{X}, \mbC) \to H^{\star}(X, \mbC),
\end{align}

via relative Gromov--Witten invariants of pairs $(X_1,D),(\widetilde{X}_2,D)$ of smooth components $X_1,\widetilde{X}_2$ of a special fiber $\widetilde{X}_0:= p^{-1}(0)$, intersecting transversally along a smooth divisor $D$.

\subsection{Relative Gromov--Witten invariants}
Let $X$ be a smooth projective complex variety with $D\subset X$ a smooth divisor and denote by $N_{X/D}$ the normal line bundle over $D$. We recall the notion of a stable relative map $(C,q_1,\dots, q_k)\to (X, D)$ from prestable nodal curve $C$ with $k$ distinct smooth marked points $q_i$ touching the divisor $D$ with multiplicities $\mu = (\mu_1,\dots, \mu_k)$. These points of $C$ are called \textit{relative points}.

First, recall the notion of $l$-step expanded degeneration $(X[l], D[l])$ of $(X,D)$. Denote by 
\begin{align*}
    \mathbf{P}=\mathbb{P}(N_{X/D}\oplus \mathcal{O}_D)
\end{align*}
the $\mathbb{P}^1$-bundle over $D$. The variety $\mathbf{P}$ has two natural disjoint sections. Choice of a point $[1:0]$ in a fiber over each point in $D$ defines a section of $\mathbf{P}$ with normal bundle $N_{X/D}^\vee$, called zero section $D_0\simeq D$. Choice of a point $[0:1]$ defines a section with normal bundle $N_{X/D}$, called infinity section $D_\infty \simeq D$, see \cite[\S 1.4.2]{ABPZ}. Gluing $X$ and $\mathbf{P}$ along $D$ and $\mathbf{P}$  along $D_0$ gives a variety $X[1]$ with a divisor $D[1]=D_\infty$. Gluing $X[1]$ along $D[1]$ with $\mathbf{P}$ along $D_0$ gives variety $X[2]$ with a divisor $D[2]=D_\infty$. The $l$-\textit{step expanded
degeneration} of the pair $(X,D)$ is the pair $(X[l], D[l])$, obtained by gluing $X[l-1]$ along $D[l-1]$ and $\mathbf{P}$ along $D_0$.

Moduli stack of stable relative maps was constructed in \cite{Li1}. It is a proper Deligne-Mumford stack \cite{Li1} and its virtual fundamental cycle was constructed in \cite{Li2}. Denote its connected component by $\overline{\mathcal{M}}^\mu_{g, n + k, \beta}(X,D)$ parametrizing maps $$[(C,p_1,\dots, p_n, q_1,\dots, q_k)\xrightarrow{f} (X[l],D[l])]$$ 
\begin{itemize}
    \item to some $l$-step expanded degeneration $(X[l],D[l])$ (the $l$ can vary inside the component);
    \item from nodal genus $g$ curves, of curve class $\beta\in H_2(X,\mbZ)$, i.e. $H_2(C,\mbZ)\xrightarrow{f_*} H_2(X[l],\mbZ)\xrightarrow{p_*} H_2(X,\mbZ)$ sends $[C]$ to $\beta$;
    \item with $k$ relative points that map to $D[l]$ with a 
  fixed vector $\mu = (\mu_1,\dots,\mu_k)\in \mbZ^k_{>0}$ of tangency conditions; 
    \item and the rest $n$ points are mapped to $X[l]_{\text{smooth}}\backslash D[l]$;
\item no irreducible source component maps into the outer relative divisor or the singular locus of the expanded target;
\item the inverse image of an internal gluing divisor consists of nodes, with matching contact orders on their two branches;
\item \(f^*D[l]=\sum_i\mu_iq_i\);
\item the automorphism group of $f$ is finite.

\end{itemize}

   It is empty if $\sum \mu_i \neq \langle \beta, D\rangle$.

Consider its open substack 
\begin{align*}
  \mathcal{M}^\mu_{g, n + k, \beta}(X, D) \subset \overline{\mathcal{M}}^\mu_{g, n + k, \beta}(X, D) 
\end{align*}
of stable relative maps from smooth curves to $(X,D)$ itself, i.e. to $0$-step expanded degeneration. It has the same virtual dimension
\begin{align*}
    \virdim~\overline{\mathcal{M}}^\mu_{g, n + k, \beta}(X, D) = 
    \virdim~ \mathcal{M}^\mu_{g, n + k, \beta}(X, D).
\end{align*}

Intersection of $D$ with the image of the curve with multiplicity $\mu_i$ is a codimension $\mu_i$ condition, so
\begin{align*}
    \virdim~\mathcal{M}^\mu_{g, n + k, \beta}(X, D) = \virdim~ \mathcal{M}_{g, n + k, \beta}(X) - \sum_{i=1}^k \mu_i,
\end{align*}
where $\mathcal{M}_{g, n + k, \beta}(X)$ is the moduli stack of stable maps from smooth curves. Finally, we get 
\begin{align}\label{virdim}
    \virdim~ \overline{\mathcal{M}}^\mu_{0, n + k, \beta}(X, D) = \dim_\mbC X -3 + \langle\beta,c_1(T_X)- D\rangle + n + k.
\end{align}

\begin{example}
Fix a vector of positive integers $a=(a_1,\dots, a_k),a_i\geq 1$ with $\sum_{i=1}^k a_i = d$.
A point of the stack 
$\overline{\mathcal{M}}^a_{0,k,d}(\mbP^1,\infty)$
is a map $[(X,q_1,\dots, q_k)\xrightarrow{f} (Y,\infty)]$, $(n=0)$, where \begin{enumerate}
    \item [$\cdot$] $X$ is a nodal curve of arithmetic genus 0; 
    \item [$\cdot$] $Y$ is a nodal curve of arithmetic genus 0 whose dual graph is a tree without branches;
    \item [$\cdot$] the point $\infty\in Y$ is smooth and lie on the extreme right component;
    \item [$\cdot$] the preimage of each node in $Y$ is a union of nodes in $X$;
    \item [$\cdot$]  lifting $f$ to the map between normalisations at each node, the ramification indices at the two preimages of each node of $X$ agree;
    \item [$\cdot$] $f^{-1}(\infty) = \{q_1,\dots, q_k\}$ and the ramification index at $q_i$ is $a_i$.
\end{enumerate}
The virtual dimension in this case is $d+k-2$ and is the number of simple branch points away from $\infty$, by Riemann-Hurwitz formula. 
\end{example}

The \textit{relative Gromov--Witten invariant} of $X$ with absolute insertions $\gamma \in H^{\star}(X^n,\mbC)$ and relative insertions $\delta \in H^\star(D^k,\mbC)$ 
  will be denoted by

\begin{align*}
\left\langle\gamma \ |\ (\mu, \delta) \right\rangle_{g,\beta}^{(X, D)}:= \int_{\left[\overline{\mathcal{M}}^\mu_{g, n + k, \beta}(X, D)\right]^{\vir}}\ev^*(\gamma)\cdot \ev_{\mathrm{rel}}^*(\delta), \\
\end{align*}
where $\ev: \overline{\mathcal{M}}^\mu_{g, n + k, \beta}(X, D)\to X^n$ sends a map $[(C,p_1,\dots, p_n, q_1, \dots, q_k)\to (X[l],D[l])]$ to $$((c_l\circ f)(p_1),\dots, (c_l\circ f)(p_n)),$$  
where $c_l: X[l]\to X$ contracts the expansion, and  $\ev_{\mathrm{rel}}: \overline{\mathcal{M}}^\mu_{g, n + k, \beta}(X, D)\to D^k$ send a map $[(C,p_1,\dots, p_n, q_1, \dots, q_k)\to (X[l],D[l])]$ to $$(f(q_1),\dots, f(q_k)),$$

with $f(p_i)\in$ $X[l]_{\text{smooth}}\backslash D[l]$ and $f(q_i)\in D[l]\simeq D$.
    
\subsection{Proof of the main theorem}

We use Jun Li's degeneration formula \cite{Li2} for the family
$p:\widetilde X\to\Delta$ with central fibre
$X_1\cup_D\widetilde X_2$. We state its genus-zero form for this
family. Let $\gamma_1,\ldots,\gamma_n$ be even cohomology classes on $X$
with lifts $\widetilde\gamma_i$ to $\widetilde X$, and put
$\gamma_i^{(1)}=\widetilde\gamma_i|_{X_1}$ and
$\gamma_i^{(2)}=\widetilde\gamma_i|_{\widetilde X_2}$.
Choose a homogeneous basis $\{\delta_a\}$ of $H^*(D,\mbC)$ and its
Poincar\'e-dual basis $\{\delta_a^\vee\}$, so that
$\int_D\delta_a\delta_b^\vee=\delta_{ab}$. Then
\begin{equation}\label{deg_formula}
\begin{aligned}
\langle\gamma_1,\ldots,\gamma_n\rangle^X_{0,n,\beta}
={}&\sum_{[\Gamma]}
\frac{\prod_{e\in E(\Gamma)}\mu_e}{|\Aut(\Gamma)|}
\sum_{(a_e)_{e\in E(\Gamma)}}\\
&\quad\prod_{v\in V_1(\Gamma)}
\langle(\gamma_i^{(1)})_{i\in I_v}\ |\
 (\mu_e,\delta_{a_e})_{e\ni v}\rangle^{(X_1,D)}_{0,\beta_v}\\
\cdot&\quad\prod_{v\in V_2(\Gamma)}
\langle(\gamma_i^{(2)})_{i\in I_v}\ |\
 (\mu_e,\delta^\vee_{a_e})_{e\ni v}\rangle^{(\widetilde X_2,D)}_{0,\beta_v}.
\end{aligned}
\end{equation}
Here are the notation and the conditions on the sum.
\begin{enumerate}
\item The first sum runs over isomorphism classes of admissible connected
bipartite graphs $\Gamma$, with vertex colours $1$ and $2$ corresponding
to $X_1$ and $\widetilde X_2$. Their vertex sets are denoted by
$V_1(\Gamma)$ and $V_2(\Gamma)$; $E(\Gamma)$ is the edge set.
A vertex $v$ is decorated by its genus $g_v$, an effective curve class
$\beta_v$ in the corresponding target component, and a subset
$I_v\subset\{1,\ldots,n\}$ of labelled absolute markings.
The sets $I_v$ form a partition. The genus condition is
\[
 0=\sum_v g_v+b_1(\Gamma).
\]
Thus every $g_v=0$ and $\Gamma$ is a tree.
\item Each vertex indexes a moduli factor of connected stable relative
maps with numerical data
$(0,|I_v|,k_v,\beta_v,(\mu_e)_{e\ni v})$, where
$k_v=\#\{e:e\ni v\}$.  A vertex does
not select an irreducible component of the moduli space, nor does it require
its source curve to be irreducible. The associated relative invariant
integrates over the full virtual fundamental class for these data.
\item Each edge joins vertices of different colours and represents two
relative markings identified in the gluing. Their evaluations in $D$
must coincide, and their common contact order is an integer $\mu_e>0$.
The diagonal class of $D$ accounts for the pair
$\delta_{a_e}$ and $\delta^\vee_{a_e}$ in the formula. At every vertex,
\[
 \langle\beta_v,D\rangle=\sum_{e\ni v}\mu_e.
\]

\item The classes $\beta_v$ satisfy the curve-class specialization
condition for $\beta$
\begin{align*}
    \sum_{v\in V_1(\Gamma)}(j_1)_*\beta_v + \sum_{v\in V_2(\Gamma)}(j_2)_*\beta_v = i_*\beta
\end{align*}
with $j_1,~j_2,~i$ the embeddings of $X_1,~ \widetilde{X}_2,~X$ into $\widetilde{X}$. Writing $d_v$ for the pairing of $\beta_v$ with $\omega_{X_1}$ if $v\in V_1(\Gamma)$ or with the pullback of $\omega_{X_2}$ to $\widetilde{X}_2$ if $v\in V_2(\Gamma)$, one has $\sum_v d_v = \langle \beta, \omega_X\rangle$.

\item In each bracket, entries to the left of the vertical bar are
absolute insertions; entries to its right specify the relative contact
orders and cohomology insertions. Relative markings are labelled by
incident edges, with no additional factorial in the bracket.
The second sum chooses a basis index $a_e$ independently for every edge.
The automorphism group preserves the colours, genera, curve classes,
contact orders and all absolute marking labels. This convention gives
the coefficient $\prod_e\mu_e/|\Aut(\Gamma)|$.
\end{enumerate}

Relative moduli spaces here include expanded targets and nodal sources.
Components mapping to the projective bundles in an expansion are already
included in a vertex factor; they do not give a third colour in
\eqref{deg_formula}.

We now specialize to $n=m+3$, degree $\beta=\frac{m}{2}$, and
$\gamma_i=e_i$, where the $e_i$ are the classes from
Proposition~\ref{restr}, with lifts as in
Corollary~\ref{n1Argument}. Thus the left-hand side of
\eqref{deg_formula} becomes
\[
 \langle e_1,\ldots,e_{m+3}\rangle^X_{0,m+3,\frac{m}{2}},
\]
and the graph conditions include
$\coprod_v I_v=\{1,\ldots,m+3\}$ and $\sum_v d_v=\frac{m}{2}$.

\begin{proof}[Proof of Theorem~\ref{maintheorem}]
We show that every graph contribution in \eqref{deg_formula} is zero.

\paragraph{The marking distribution.}
Corollary~\ref{n1Argument} gives
\[
 \widetilde e_i|_{X_1}=0\quad i\ne m+2,
 \qquad \widetilde e_{m+2}|_{\widetilde X_2}=0.
\]
Consequently, in any potentially nonzero term, the marking $m+2$
must lie on the $X_1$ side, and every other marking must lie on the
$\widetilde X_2$ side. There is exactly one marked $X_1$-vertex;
it carries $I_v=\{m+2\}$. Every other $X_1$-vertex is unmarked.
Both colours occur, so the connected graph has edges and every vertex
has at least one incident edge. A single-vertex graph with no edges
has a zero insertion, whichever colour that vertex has, and therefore
contributes zero.

\paragraph{Dimension at an $X_1$-vertex.}
For $v\in V_1(\Gamma)$, set
\[
 n_v=|I_v|\in\{0,1\},\qquad
 k_v=\#\{e:e\ni v\},\qquad
 d_v=\int_{\beta_v}\omega_{X_1}.
\]
Since $D$ is a hyperplane section of $X_1$, the contact condition implies
\begin{equation}\label{ineq2}
 d_v=\sum_{e\ni v}\mu_e\geq k_v\geq1.
\end{equation}
This uses positivity of the integer contact orders; it does not assume
that all contacts are transverse. The normal sequence for the quadric
$X_1\subset\mbP^{m+1}$ gives $c_1(T_{X_1})=m\omega_{X_1}$.
Using \eqref{virdim}, the complex virtual dimension of this connected
relative moduli factor is
\begin{equation}\label{virdim_X1}
\begin{aligned}
 \virdim~\overline{\mathcal M}^{\mu}_{0,n_v+k_v,\beta_v}(X_1,D)
 &=m-3+md_v+n_v+k_v-\sum_{e\ni v}\mu_e\\
 &=m-3+(m-1)d_v+n_v+k_v.
\end{aligned}
\end{equation}
Because $D$ is an odd-dimensional quadric, its cohomology is entirely
even. Write $\operatorname{codim}_{\mbC}\delta_a=\frac12\deg\delta_a$;
then
\[
 0\leq\operatorname{codim}_{\mbC}\delta_a\leq m-1.
\]
The lower endpoint is zero because the basis includes the unit.
Every absolute insertion at $v$ has complex codimension $\frac{m}{2}$.
For a nonzero numerical relative invariant, one must therefore have
\begin{equation}\label{nonzero_iff}
\begin{aligned}
 m-3+(m-1)d_v+n_v+k_v
 &=\frac m2 n_v+\sum_{e\ni v}\operatorname{codim}_{\mbC}\delta_{a_e}\\
 &\leq\frac m2 n_v+(m-1)k_v.
\end{aligned}
\end{equation}
Combining \eqref{ineq2} and \eqref{nonzero_iff} gives the necessary bound
\begin{equation}\label{ineq1}
 k_v\leq3-m+\left(\frac m2-1\right)n_v.
\end{equation}

\paragraph{Eliminating the graphs.}
If an $X_1$-vertex is unmarked, $n_v=0$ and \eqref{ineq1} gives
$k_v\leq3-m<0$, which is impossible for $m\geq4$.
Every graph with such a vertex contributes zero.
Thus the only remaining candidates have exactly one $X_1$-vertex,
carrying $e_{m+2}$. Since the graph is a bipartite tree, it must be a
star: its central $X_1$-vertex is joined by one edge to each
$\widetilde X_2$-vertex, and the other markings are distributed among
these leaves. 

For the central vertex, $n_v=1$, and \eqref{ineq1} becomes
\[
 1\leq k_v\leq2-\frac m2.
\]
This is impossible for every even $m\geq4$. In particular, when $m=4$
the upper bound is zero, while the marking distribution has already
shown that $k_v\geq1$. Thus the star graphs also contribute zero,
including the graph with two vertices and one edge.
There is no surviving graph. The degeneration formula yields
\[
 \langle e_1,\ldots,e_{m+3}\rangle^X_{0,m+3,\frac{m}{2}}=0.
\]

\end{proof}

\end{document}